\documentclass{amsart}%
\usepackage{amsfonts}
\usepackage{amsmath}
\usepackage{amssymb}
\usepackage{graphicx}%
\providecommand{\U}[1]{\protect\rule{.1in}{.1in}}
\newtheorem{theorem}{Theorem}
\theoremstyle{plain}

\newtheorem{corollary}{Corollary}

\newtheorem{lemma}{Lemma}

\newtheorem{proposition}{Proposition}
\newtheorem{remark}{Remark}

\numberwithin{equation}{section}
\begin{document}
\title[Lancaster expansions]{On a class of bivariate distributions built of $q$-ultraspherical polynomials.}
\author{Pawe\l \ J. Szab\l owski}
\address{Department of Mathematics and Information Sciences\\
Warsaw University of Technology\\
ul Koszykowa 75, 00-662 Warsaw, Poland}
\email{pawel.szablowski@gmail.com}
\date{September, 2022}

\begin{abstract}
Our primary result concerns the positivity of specific kernels constructed
using the $q$-ultraspherical polynomials. In other words, it concerns a
two-parameter family of bivariate, compactly supported distributions.
Moreover, this family has a property that all its conditional moments are
polynomials in the conditioning random variable. The significance of this
result is evident for individuals working on distribution theory, orthogonal
polynomials, $q$-series theory, and the so-called quantum polynomials.
Therefore, it may have a limited number of interested researchers. That is
why, we put our results into a broader context. We recall the theory of
Hilbert-Schmidt operators and the idea of Lancaster expansions of the
bivariate distributions absolutely continuous with respect to the product of
their marginal distributions. Applications of Lancaster expansion can be found
in Mathematical Statistics or the creation of Markov processes with polynomial
conditional moments (the most well-known of these processes is the famous
Wiener process).

\end{abstract}
\maketitle

\section{Introduction}

As stated in the Abstract our main result concerns the positivity of certain
bivariate kernels built of the so-called $q$-ultraspherical polynomials. The
proof is long and difficult involving certain identities from the so-called
$q$-series theory. That is why, we will present our results in a broader
context that includes, Hilbert-Schmidt operators, kernels built of orthogonal
polynomials, Markov processes with polynomial conditional moments and
bivariate distributions that are absolutely continuous with respect to the
products of their marginals and their applications in Mathematical statistics.

We will be dealing mostly with expressions of type
\begin{equation}
K(x,y)=\sum_{j=0}^{\infty}c_{j}a_{j}(x)b_{j}(y)\text{,} \label{kerr}%
\end{equation}
where $\left\{  a_{j}(x)\right\}  $ and $\left\{  b_{j}(y)\right\}  $ are the
sets of real polynomials orthogonal with respect to some finite real measures
respectively $d\alpha(x)$ and $d\beta(y)$. Since the terminology found in the
literature is somewhat confusing, we will call such expressions kernels,
symmetric when the two measures are the same (and consequently families
$\left\{  a_{j}(x)\right\}  $ and $\left\{  b_{j}(y)\right\}  $ are the same)
or non-symmetric when the two measures are different.

The most general application of such kernels are the so-called Hilbert-Schmidt
operators considered in functional analysis. Namely, imagine that we have two
Hilbert, real (for the sake of simplicity of argument) spaces $L_{2}%
(d\alpha(x))$ and $L_{2}(d\beta(y))$ of functions that are square integrable
with respect to respectively $d\alpha$ and $d\beta$. Let $\left\{
a_{j}(x)\right\}  $ and $\left\{  b_{j}(y)\right\}  $ be the orthonormal bases
of these spaces. Let us take a function $h$ say, from $L_{2}(d\alpha(x))$. It
can be presented in the following form $h(x)\allowbreak=\allowbreak\sum
_{n\geq0}h_{n}a_{n}(x),$ with $\sum_{n\geq0}h_{n}^{2}<\infty$. Then
\[
f(y)=\int K(x,y)h\left(  x\right)  d\alpha(x)=\sum_{n\geq0}c_{n}h_{n}%
b_{n}(y)\text{.}%
\]
We can observe that if only, $\sum_{n\geq0}c_{n}^{2}<\infty$ then $f\left(
y\right)  \in L_{2}(d\beta(y))$. Moreover, $K(x,y)$ is the kernel defining a
Hilbert--Schmidt operator with the norm equal $\sum_{n\geq0}c_{n}^{2}$.

Now, assume that both families of polynomials $\left\{  a_{j}\right\}  $ and
$\left\{  b\right\}  $ are orthonormal, i.e., $\int a_{j}^{2}d\alpha
\allowbreak=\allowbreak\int b_{j}^{2}d\beta\allowbreak=\allowbreak1$. Then the
condition $\sum_{n\geq0}c_{n}^{2}<\infty$ implies $\sum_{n\geq0}c_{n}^{2}%
a_{n}^{2}\left(  x\right)  <\infty$ almost everywhere mod $d\alpha$ and
$\sum_{n\geq0}c_{n}^{2}b_{n}^{2}\left(  x\right)  <\infty$ almost everywhere
mod $d\beta$. Consequently, for almost all $x\in\operatorname*{supp}\alpha$
$:$ $\sum_{j=0}^{\infty}c_{j}a_{j}(x)b_{j}(y)$ is convergent in mean squares
sense with respect to $d\beta$ (as a function of $y)$. Similarly for $x$ and
$y$ interchanged.

The kernel $K(x,y)$ is positive, if it is non-negative for almost all $x$ mod
$d\alpha$ and $y$ mod $d\beta$. Positive kernels define stationary Markov
processes whose all its conditional moments are polynomials in the
conditioning random variable. How to do it see, e.g., \cite{SzablPoly},
\cite{SzabStac} and \cite{Szab17}. For more introduction and the list of
literature treating this subject was done in the recently published in
\textit{Stochastics }paper \cite{Szab22}. Let us also remark that following
\cite{Szabl21}, positive kernels, scaled to $1$ constitute the totality of
bivariate distributions satisfying condition (\ref{ms}) presented below, whose
conditional moments of, say order $n$, are the polynomials of order not
exceeding $n$ in the conditioning random variable.

The main problem is to present kernels in the compact, simple forms. The
process of doing so is called summing the kernels. It turns out to be a very
hard problem. Only a few kernels are summed. A sample of a few positive,
summed kernels will be presented in Section \ref{lista}.

We will present several compactly supported, bivariate distributions allowing
the so-called Lancaster expansion. In the series of papers \cite{Lancaster58},
\cite{Lancaster63(1)}, \cite{Lancaster63(2)}, and \cite{Lancaster75}, H.
Lancaster analyzed the family of bivariate distributions which allow a special
type of expansion of its Radon-Nikodym derivative of these distributions with
respect to the products of their marginal measures. To be more precise, let's
assume that $d\mu(x,y)$ is the distribution in question and $d\alpha(x)$ and
$d\beta(x)$ are its marginal distributions, and the following condition is met%
\begin{equation}
\int\int\left(  \frac{\partial^{2}\mu(x,y)}{\partial\alpha(x)\partial\beta
(y)}\right)  ^{2}d\alpha(x)d\beta(y)<\infty\text{,} \label{ms}%
\end{equation}
where the integration is over the support of the product measure $d\alpha
d\beta$. It turns out that then the following expansion is convergent at least
in mean square with respect to the product measure:%
\begin{equation}
d\mu(x,y)=d\alpha(x)d\beta(y)\sum_{j=0}^{\infty}c_{j}a_{j}(x)b_{j}(y)\text{.}
\label{Lan}%
\end{equation}

From the theory of orthogonal series, it follows, that if the condition
(\ref{ms}) is satisfied and we deal with expansion (\ref{Lan}) convergent in
mean square, then we must have
\[
\sum_{n\geq0}c_{n}^{2}<\infty\text{,}%
\]
that is $\left\{  c_{n}\right\}  _{n\geq0}\in l_{2}$, the space of square
summable sequences. When equipped in the following norm $\left\Vert
\mathbf{c}\right\Vert \allowbreak=\allowbreak\sum_{n\geq0}c_{n}^{2}$ where:
$\mathbf{c}\allowbreak=\allowbreak\left\{  c_{n}\right\}  _{n\geq0}$, then
$l_{2}$ becomes a Banach space.

In the expansion (\ref{Lan}), $\left\{  a_{j}\right\}  $ and $\left\{
b_{j}\right\}  $ are orthonormal sequences with respect to respectively
$d\alpha$ and $d\beta$. We will call expansions of the form (\ref{Lan})
satisfying a condition (\ref{ms}) Lancaster expansion, briefly LE. One has to
notice, that expressions like (\ref{Lan}), more precisely of the form%
\[
\sum_{j=0}^{\infty}c_{j}a_{j}(x)b_{j}(y)\text{,}%
\]
where $\left\{  a_{j}\right\}  $ and $\left\{  b_{j}\right\}  $ are two
families of polynomials are kernels, discussed above.

Obviously, for probabilists the most interesting are those kernels that are
nonnegative for certain ranges of $x$ and $y$. Nonnegative kernels on some
subsets of $\mathbb{R}^{2}$ will be called Lancaster kernels briefly LK. By
scaling properly, one can use such a kernel to construct a positive, bivariate
measure. In the probabilistic context of this paper, 'LE' will serve as the
most precise and suitable term. Probably the first LE expansion was the
following one:%
\begin{gather*}
\exp(-(x^{2}+y^{2})/2)\sum_{j\geq0}^{\infty}\frac{\rho^{j}}{j!}H_{j}%
(x)H_{j}(y)=\\
\exp(-(x^{2}-2\rho xy+y^{2})/(2(1-\rho^{2}))/(2\pi\sqrt{1-\rho^{2}})\text{,}%
\end{gather*}
that is convergent for all $x,y\in\mathbb{R}$ and $\left\vert \rho\right\vert
<1$. Above, $H_{j}$ denotes the $j$-th element of the family of the so-called
probabilistic Hermite polynomials, described below in Section \ref{poly} (See
\cite{Mercier09}). As mentioned earlier, in Section \ref{lista}, we will list
all known to the author such LE.

Why such expansions are important? Lancaster himself a long time ago pointed
out applications in mathematical statistics, hence we will not repeat these
arguments. In this paper, we will concentrate on applications in distribution
theory by indicating bivariate distributions having a simple structure. As
shown in \cite{Szabl21} a bivariate distribution that allows LE and satisfies
the condition (\ref{ms}) has the property that all its conditional moments of
degree say $n$, are polynomials of order $n$ in the conditioning random
variable. More precisely, for all random variables $(X,Y)$ having bivariate
distribution given by (\ref{Lan}), we have for all $n\geq0:$%
\[
E(a_{n}(X)|Y=y)=c_{n}b_{n}(y)\text{.}%
\]

We have an immediate observation concerning coefficients $\left\{
c_{n}\right\}  _{n\geq0}$ that appear in the definitions of LK. Observe that
if we assume that LK must be a probability distribution (i.e. integrating to
$1$) then the coefficients $\left\{  c_{n}\right\}  _{n\geq0}$ form a convex
cone in the space $l_{2}$ of square summable sequences.

Other possible applications are in the theory of Markov processes. Namely,
recall, that every Markov process $\left\{  X_{t}\right\}  _{t\in\mathbb{R}}$
is completely defined by the two families of measures. The first one is the
family of the so-called marginal distributions, i.e. family (indexed by the
time parameter $t)$ of one-dimensional distributions of a random variable
$X_{t}$. The other family of distributions is the family of the conditional
distributions of $X_{t+\tau}|X_{\tau}=x$, indexed by $t\geq0$ and $x$. One
could utilize symmetric LE almost immediately by finding positive constants
$\gamma_{n}$ such that $\exp(-t\gamma_{n})\allowbreak=\allowbreak c_{n}$ as it
was done e.g. in \cite{Szab22} and define a stationary Markov process.

Many examples of LE stem from the application of the polynomials from the
so-called Askey-Wilson family of polynomials. These polynomials involve
notions of the so-called $q$-series theory, so in the next section we will
present basic notions and facts from this theory. The traditional terminology
calls polynomial appearing within $q$-series theory, quantum polynomials, see,
e.g., an excellent monograph \cite{IA} of Ismail.

The paper is arranged as follows. The next Section \ref{Notation} includes the
traditional notation used in the $q$-series theory and some general results
used in the subsequent section. The next Section \ref{poly} contains a list of
polynomials mostly from the so-called Askey--Wilson (AW) scheme. It is
important to present these polynomials and their relationship to the
$q$-ultraspherical that is the main subject of the paper. The next Section
\ref{main} is dedicated to the proof of our main result i.e. summing and
proving the positivity of certain kernel built of $q$-ultraspherical
polynomials. The next Section \ref{lista} lists simple, known to the author,
summed kernels both symmetric and non-symmetric kernels. Finally, the last
Section \ref{aux1} contains longer, requiring tedious calculations, proofs and
other auxiliary results from $q$-series theory.

\section{Notation, definition and some methods of obtaining LE
\label{Notation}}

$q$ is a parameter. We will assume that $-1<q\leq1$ unless otherwise stated.
The case $q\allowbreak=\allowbreak1$ may not always be considered directly,
but sometimes as left-hand side limit (i.e., $q\longrightarrow1^{-}$). We will
point out these cases.

We will use traditional notations of the $q$-series theory, i.e.,%
\[
\left[  0\right]  _{q}\allowbreak=\allowbreak0,~\left[  n\right]
_{q}\allowbreak=\allowbreak1+q+\ldots+q^{n-1}\allowbreak,\left[  n\right]
_{q}!\allowbreak=\allowbreak\prod_{j=1}^{n}\left[  j\right]  _{q},\text{with
}\left[  0\right]  _{q}!\allowbreak=1\text{,}%
\]%
\[%
\genfrac{[}{]}{0pt}{}{n}{k}%
_{q}=%
\begin{cases}
\frac{\left[  n\right]  _{q}!}{\left[  n-k\right]  _{q}!\left[  k\right]
_{q}!}, & \text{if }n\geq k\geq0\text{ ;}\\
0, & \text{otherwise.}%
\end{cases}
\]
$\binom{n}{k}$ will denote the ordinary, well-known binomial coefficient.

It is useful to use the so-called $q$-Pochhammer symbol for $n\geq1$%
\[
\left(  a|q\right)  _{n}=\prod_{j=0}^{n-1}\left(  1-aq^{j}\right)  ,~~\left(
a_{1},a_{2},\ldots,a_{k}|q\right)  _{n}\allowbreak=\allowbreak\prod_{j=1}%
^{k}\left(  a_{j}|q\right)  _{n}\text{,}%
\]
with $\left(  a|q\right)  _{0}\allowbreak=\allowbreak1$.

Often $\left(  a|q\right)  _{n}$ as well as $\left(  a_{1},a_{2},\ldots
,a_{k}|q\right)  _{n}$ will be abbreviated to $\left(  a\right)  _{n}$ and
\newline$\left(  a_{1},a_{2},\ldots,a_{k}\right)  _{n}$, if it will not cause misunderstanding.

\begin{remark}
In the literature functions also an ordinary Pochhammer symbol, i.e.,
$a(a+1)\ldots(a+n-1)$. We will denote it by $\left(  a\right)  ^{(n)}$ and
call "rising factorial". There, in the literature, functions also the
so-called "falling factorial" equal to $a(a-1)\ldots(a-n+1)$ that we will
denote $\left(  a\right)  _{\left(  n\right)  }$. Hence, in this paper
$\left(  a\right)  _{n}$ would mean $\left(  a|q\right)  _{n}$ as defined above.
\end{remark}

We will also use the following symbol $\left\lfloor n\right\rfloor $ to denote
the largest integer not exceeding $n$.

For further reference we mention the following four formulae from
\cite{KLS}(Subsections 1.8 1.14). Namely, the following formulae are true for
$\left\vert t\right\vert <1$, $\left\vert q\right\vert <1$ (already proved by
Euler, see \cite{Andrews1999} Corollary 10.2.2)
\begin{align}
\frac{1}{(t)_{\infty}}\allowbreak &  =\allowbreak\sum_{k\geq0}\frac{t^{k}%
}{(q)_{k}}\text{, }\frac{1}{(t)_{n+1}}=\sum_{j\geq0}%
\genfrac{[}{]}{0pt}{}{n}{j}%
_{q}t^{j}\text{,}\label{binT}\\
(t)_{\infty}\allowbreak &  =\allowbreak\sum_{k\geq0}(-1)^{k}q^{\binom{k}{2}%
}\frac{t^{k}}{(q)_{k}}\text{, }\left(  t\right)  _{n}=\sum_{j=0}^{n}%
\genfrac{[}{]}{0pt}{}{n}{j}%
_{q}q^{\binom{j}{2}}(-t)^{j}\text{.} \label{naw}%
\end{align}

It is easy to see, that $\left(  q\right)  _{n}=\left(  1-q\right)
^{n}\left[  n\right]  _{q}!$ and that%
\[%
\genfrac{[}{]}{0pt}{}{n}{k}%
_{q}\allowbreak=\allowbreak\allowbreak%
\begin{cases}
\frac{\left(  q\right)  _{n}}{\left(  q\right)  _{n-k}\left(  q\right)  _{k}%
}\text{,} & \text{if }n\geq k\geq0\text{;}\\
0\text{,} & \text{otherwise.}%
\end{cases}
\]
\newline The above-mentioned formula is just an example where direct setting
$q\allowbreak=\allowbreak1$ is senseless however, the passage to the limit
$q\longrightarrow1^{-}$ makes sense.

Notice, that, in particular,
\begin{equation}
\left[  n\right]  _{1}\allowbreak=\allowbreak n\text{,}~\left[  n\right]
_{1}!\allowbreak=\allowbreak n!\text{,}~%
\genfrac{[}{]}{0pt}{}{n}{k}%
_{1}\allowbreak=\allowbreak\binom{n}{k},~(a)_{1}\allowbreak=\allowbreak
1-a\text{,}~\left(  a|1\right)  _{n}\allowbreak=\allowbreak\left(  1-a\right)
^{n} \label{q1}%
\end{equation}
and
\begin{equation}
\left[  n\right]  _{0}\allowbreak=\allowbreak%
\begin{cases}
1\text{,} & \text{if }n\geq1\text{;}\\
0\text{,} & \text{if }n=0\text{.}%
\end{cases}
\text{,}~\left[  n\right]  _{0}!\allowbreak=\allowbreak1\text{,}~%
\genfrac{[}{]}{0pt}{}{n}{k}%
_{0}\allowbreak=\allowbreak1\text{,}~\left(  a|0\right)  _{n}\allowbreak
=\allowbreak%
\begin{cases}
1\text{,} & \text{if }n=0\text{;}\\
1-a\text{,} & \text{if }n\geq1\text{.}%
\end{cases}
\label{q2}%
\end{equation}

$i$ will denote imaginary unit, unless otherwise stated. Let us define also:%

\begin{align}
\left(  ae^{i\theta},ae^{-i\theta}\right)  _{\infty}  &  =\prod_{k=0}^{\infty
}v\left(  x|aq^{k}\right)  \text{,}\label{rozklv}\\
\left(  te^{i\left(  \theta+\phi\right)  },te^{i\left(  \theta-\phi\right)
},te^{-i\left(  \theta-\phi\right)  },te^{-i\left(  \theta+\phi\right)
}\right)  _{\infty}  &  =\prod_{k=0}^{\infty}w\left(  x,y|tq^{k}\right)
\text{,}\label{rozklw}\\
\left(  ae^{2i\theta},ae^{-2i\theta}\right)  _{\infty}  &  =\prod
_{k=0}^{\infty}l\left(  x|aq^{k}\right)  \text{,} \label{rozkll}%
\end{align}
where,
\begin{align}
v(x|a)\allowbreak &  =\allowbreak1-2ax+a^{2}\text{,}\label{vxa}\\
l(x|a)  &  =(1+a)^{2}-4x^{2}a\text{,}\label{lsa}\\
w(x,y|a)  &  =(1-a^{2})^{2}-4xya(1+a^{2})+4a^{2}(x^{2}+y^{2}) \label{wxya}%
\end{align}
and, as usually in the $q$-series theory, $x\allowbreak=\allowbreak\cos\theta$
and $y=\cos\phi$.

We will use also the following notation:%
\[
S(q)\overset{df}{=}%
\begin{cases}
\lbrack-2/\sqrt{1-q},2/\sqrt{1-q}], & \text{if }\left\vert q\right\vert
<1\text{;}\\
\mathbb{R}, & \text{if }q=1\text{.}%
\end{cases}
\]

\subsection{Method of expansion of the ratio of densities}

We will use through the paper the following way of obtaining infinite
expansions of type
\[
\sum_{j\geq0}d_{n}p_{n}(x)\text{,}%
\]
that are convergent almost everywhere on some subset of $\mathbb{R}$. Namely,
in view of \cite{Szab4} let us consider two measures on $\mathbb{R}$ both
having densities $f$ and $g$. Furthermore, suppose that, we know that
$\int(f(x)/g(x))^{2}g(x)dx$ is finite. Further suppose also that we know two
families of orthogonal polynomials $\left\{  a_{n}\right\}  $ and $\left\{
b_{n}\right\}  $, such that the first one is orthogonal with respect to the
measure having the density $f$ and the other is orthogonal with respect to the
measure having the density $g$. Then we know that $f/g$ can be expanded in an
infinite series
\begin{equation}
\sum_{n\geq0}d_{n}b_{n}(x)\text{,} \label{expan}%
\end{equation}
that is convergent in $L^{2}(\mathbb{R},g)$. We know in particular, that
$\sum_{n\geq0}\left\vert d_{n}\right\vert ^{2}<\infty$. If additionally
\[
\sum_{n\geq0}\left\vert d_{n}\right\vert ^{2}\log^{2}(n+1)<\infty\text{,}%
\]
then by the Rademacher--Meshov theorem, we deduce that the series in question
converges not only in $L^{2}$, but also almost everywhere with respect to the
measure with the density $g$.

Thus, we will get the condition $\sum_{n\geq0}\left\vert d_{n}\right\vert
^{2}<\infty$ satisfied for free. Moreover, in many cases we will have
$\left\vert d_{n}\right\vert ^{2}\leq r^{n}$ for some $r<1$. Hence the
condition $\sum_{n\geq0}\left\vert d_{n}\right\vert ^{2}\log^{2}(n+1)<\infty$
is also naturally satisfied. If one knows the connection coefficients between
the families $\left\{  b_{n}\right\}  $ and $\left\{  a_{n}\right\}  $, i.e.,
a set of coefficients $\left\{  c_{k,n}\right\}  _{n\geq1,0\leq k\leq n}$
satisfying%
\[
b_{n}\allowbreak(x)=\allowbreak\sum_{k=0}^{n}c_{k,n}a_{k}(x)\text{,}%
\]
then $d_{n}\allowbreak=\allowbreak c_{0,n}/\int b_{n}^{2}(x)g(x)dx$. We will
refer to this type of reasoning as D(ensity) E(expansion) I(idea) (*,*) (that
is DEI(*,*)), where the first stars point out to the formula for the
connection coefficient and the second star to the formula for $\int b_{n}%
^{2}(x)g(x)dx$.

\section{Families of polynomials appearing in the paper including those
forming part of the Askey-Wilson scheme\label{poly}}

All families of polynomials listed in this section are described in many
positions of literature starting from \cite{AW85}, \cite{Andrews1999},
\cite{IA}. However, as it was noticed by the author in \cite{Szab6}, by
changing the parameters to complex conjugate and changing the usual range of
all variables from $[-1,1]$ to $S(q)$, we obtain polynomials from the AW
scheme suitable for probabilistic applications. Recently, in the review paper
\cite{Szab2020} and a few years earlier in \cite{Szab-bAW} the author
described and analyzed the polynomials of this scheme with conjugate complex
parameters. Thus, we will refer to these two papers for details.

The families of orthogonal polynomials will be identified by their three-term
recurrences. Usually, the polynomials mentioned in such a three-term
recurrence will be monic (i.e., having a coefficient of the highest power of
the variable equal to $1$). The cases when the given three-term recurrence
leads to non-monic polynomials will be clearly pointed out. Together with the
three-term recurrence we will mention the measure, usually having density,
that makes a given family of polynomials orthogonal.

In order not to allow the paper to be too large, we will mention only basic
the properties of the polynomials involved. More properties and relationships
between used families of polynomials could be found in already mentioned
fundamental positions of literature like \cite{AW85}, \cite{Andrews1999},
\cite{IA} or \cite{KLS}. One has to remark that the polynomials of the AW
scheme are presented in their basic versions where all ranges of the variables
are confined to the segment $[-1,1]$ in these positions in the literature. As
mention before, for the probabilistic applications more useful are versions
where variables range over $S(q)$. Then it is possible to pass with $q$ to $1$
and compare the results with the properties of Hermite polynomial and Normal
distribution which are the reference points to all distribution comparisons in
probability theory.

We recall these families of polynomials for the sake of completeness of the paper.

\subsection{Chebyshev polynomials}

They are of two types denoted by $\left\{  T_{n}\right\}  $ and $\left\{
U_{n}\right\}  $ called respectively of the first and second kind satisfying
the same three-term recurrence, for $n\geq1$%
\[
2xU_{n}(x)=U_{n+1}(x)+U_{n-1}(x),
\]
with different initial conditions $T_{0}(x)=1=U_{0}(x)$ and $T_{1}(x)=x$ and
$U_{1}(x)=2x$. Obviously, they are not monic. They are orthogonal respectively
with respect to arcsine distribution with the density $f_{T}(x)\allowbreak
=\allowbreak\frac{1}{\pi\sqrt{1-x^{2}}}$ and to the semicircle or Wigner
distribution with the density $f_{U}(x)\allowbreak=\allowbreak\frac{2}{\pi
}\sqrt{1-x^{2}}$. Besides we have also the following orthogonal relationships:%
\begin{align*}
\int_{-1}^{1}T_{n}(x)T_{m}(x)f_{T}(x)dx\allowbreak &  =\allowbreak%
\begin{cases}
0, & \text{if }m\neq n\text{;}\\
1\text{,} & \text{if }m=n=0\text{;}\\
2\text{,} & \text{if }m=n>0\text{.}%
\end{cases}
\\
\int_{-1}^{1}U_{n}(x)U_{m}(x)f_{U}(x)dx\allowbreak &  =\allowbreak%
\begin{cases}
0\text{,} & \text{if }m\neq n\text{;}\\
1\text{,} & \text{if }m=n\text{.}%
\end{cases}
\end{align*}

More about their properties one can read in \cite{Mason2003}.

\subsection{Hermite polynomials}

We will consider here only the so-called probabilistic Hermite polynomials
namely the ones satisfying the following three-term recurrence:%
\[
H_{n+1}(x)=xH_{n}(x)-nH_{n-1}(x),
\]
with initial conditions $H_{0}(x)=1$, $H_{1}(x)\allowbreak=\allowbreak x$.
They are monic and orthogonal with respect to the Normal $N(0,1)$ distribution
with the well-known density $f_{N}(x)\allowbreak=\allowbreak\frac{1}%
{\sqrt{2\pi}}\exp(-x^{2}/2)$. They satisfy the following orthogonal
relationship:%
\begin{equation}
\int_{-\infty}^{\infty}H_{n}(x)H_{m}(x)f_{N}(x)dx=%
\begin{cases}
0\text{,} & \text{if }m\neq n\text{;}\\
n!\text{,} & \text{if }m=n\text{.}%
\end{cases}
\label{HH}%
\end{equation}

\subsection{$q$-Hermite polynomials}

The following three-term recurrence defines the $q$-Hermite polynomials, which
will be denoted $H_{n}(x|q)$:%
\begin{equation}
xH_{n}\left(  x|q\right)  =H_{n+1}\left(  x|q\right)  \allowbreak+\left[
n\right]  _{q}H_{n-1}\left(  x\right)  \text{,} \label{3trH}%
\end{equation}
for $n\geq1$ with $H_{-1}\left(  x|q\right)  \allowbreak=\allowbreak0$,
$H_{1}\left(  x|q\right)  \allowbreak=\allowbreak1$. Notice, that now
polynomials $H_{n}(x|q)$ are monic and also that
\[
\lim_{q\rightarrow1^{-}}H_{n}(x|q)=H_{n}(x)\text{.}%
\]

Let us define the following nonnegative function where we denoted
\begin{equation}
f_{h}\left(  x|q\right)  =\frac{2\left(  q\right)  _{\infty}\sqrt{1-x^{2}}%
}{\pi}\prod_{k=1}^{\infty}l\left(  x|q^{k}\right)  \text{,} \label{fh}%
\end{equation}
with, as before, $l(x|a)\allowbreak=\allowbreak(1+a)^{2}-4x^{2}a$ and let us
define a new density
\begin{equation}
f_{N}\left(  x|q\right)  \allowbreak=\allowbreak%
\begin{cases}
\sqrt{1-q}f_{h}(x\sqrt{1-q}/2|q)/2\text{,} & \text{if }\left\vert q\right\vert
<1\text{;}\\
\exp\left(  -x^{2}/2\right)  /\sqrt{2\pi}\text{,} & \text{if }q=1\text{.}%
\end{cases}
\label{fN}%
\end{equation}
Notice that $f_{N}$ is non-negative if only $x\in S(q)$. It turns out that we
have the following orthogonal relationship:%
\begin{equation}
\int_{S(q)}H_{n}(x|q)H_{m}(x|q)f_{N}(x|q)dx=%
\begin{cases}
0\text{,} & \text{if }n\neq m\text{;}\\
\left[  n\right]  _{q}!\text{,} & \text{if }n=m\text{.}%
\end{cases}
\label{anH}%
\end{equation}

Notice, that if $X\sim f_{N}(x|q)$, then as $H_{1}(x|q)\allowbreak=\allowbreak
x$, $H_{2}(x|q)\allowbreak=\allowbreak x^{2}-1$, we deduce that $EX\allowbreak
=\allowbreak0$ and $EX^{2}\allowbreak=\allowbreak1$.

\subsection{Big $q$-Hermite polynomials}

More on these polynomials can be read in \cite{Szab2020}, \cite{Szab-rev} or
\cite{Szab-bAW}. Here, we will only mention that these polynomials are, for
example, defined by the relationship:%
\[
H_{n}(x|a,q)\allowbreak=\sum_{j=0}^{n}%
\genfrac{[}{]}{0pt}{}{n}{j}%
_{q}q^{\binom{j}{2}}(-a)^{j}H_{n-j}\left(  x|q\right)  \text{,}%
\]
where $H_{n}$ denotes a continuous $q$-Hermite polynomials, defined above and
$a\in(-1,1)$. They satisfy the following three-term recurrence:%
\[
xH_{n}(x|a,q)=H_{n+1}(x|a,q)+aq^{n}H_{n}(x|a,q)+[n]_{q}H_{n-1}(x|a,q)\text{,}%
\]
with $H_{-1}(x|a,q)\allowbreak=\allowbreak0$ and $H_{0}(x|a,q)\allowbreak
=\allowbreak1$. It is known in particular that the characteristic function of
the polynomials $\left\{  H_{n}\right\}  $ for $\left\vert q\right\vert <1$ is
given by the formula:%
\[
\sum_{n=0}^{\infty}\frac{t^{n}}{\left[  n\right]  _{q}!}H_{n}\left(
x|q\right)  =\varphi\left(  x|t,q\right)  \text{,}%
\]
where
\begin{equation}
\varphi\left(  x|t,q\right)  \allowbreak=\allowbreak\frac{1}{\prod
_{k=0}^{\infty}\left(  1-(1-q)xtq^{k}+(1-q)t^{2}q^{2k}\right)  }\text{.}
\label{chH}%
\end{equation}
Notice, that it is convergent for $t$ such that $\left\vert t\sqrt
{1-q}\right\vert <1$ and $x\in S(q)$. These polynomials satisfy the following
orthogonality relationship:
\[
\int_{S\left(  q\right)  }H_{n}\left(  x|a,q\right)  H_{m}\left(
x|a,q\right)  f_{bN}\left(  x|a,q\right)  dx\allowbreak=\allowbreak\left[
n\right]  _{q}!\delta_{m,n}\text{,}%
\]
where
\[
f_{bN}(x|a,q)=f_{N}(x|q)\varphi(x|a,q).
\]

There exists one more interesting relationship between $q$-Hermite and big
$q$-Hermite polynomials. Namely, following paper of Carlitz \cite{Carlitz72}
we get%
\[
H_{n}(x|a,q)\sum_{j\geq0}\frac{a^{j}}{\left[  j\right]  _{q}!}H_{j}%
(x|q)=\sum_{j\geq0}\frac{a^{j}}{\left[  j\right]  _{q}!}H_{j+n}(x|q).
\]
Note that this expansion nicely compliments and generalizes the results of
Proposition \ref{aux3}.

\subsection{Continuous $q$-ultraspherical polynomials}

These polynomials were first considered by Rogers in 1894 (see \cite{Rogers1},
\cite{Rogers2}, \cite{Rogers3}). They were defined for $\left\vert
x\right\vert \leq1$ by the three-term recurrence given in, e.g.,
\cite{KLS}(14.10.19). We have the celebrated connection coefficient formula
for the Rogers polynomials see \cite{IA},(13.3.1), that will be of use in the
sequel, of course after proper rescaling.%

\begin{equation}
C_{n}\left(  x|\gamma,q\right)  \allowbreak=\allowbreak\sum_{k=0}%
^{\left\lfloor n/2\right\rfloor }\frac{\beta^{k}\left(  \gamma/\beta\right)
_{k}\left(  \gamma\right)  _{n-k}\left(  1-\beta q^{n-2k}\right)  }%
{(q)_{k}\left(  \beta q\right)  _{n-k}\left(  1-\beta\right)  }C_{n-2k}\left(
x|\beta,q\right)  \text{.} \label{CnaC}%
\end{equation}

We will consider polynomials $\left\{  C_{n}\right\}  $, with a different
scaling of the variable $x\in S(q)$ and parameter $\beta\in\lbrack-1,1]$. We
have
\begin{equation}
R_{n}(x|\beta,q)=\left[  n\right]  _{q}!C_{n}(x\sqrt{1-q}/2|\beta
,q)(1-q)^{n/2}\text{,} \label{CnaR}%
\end{equation}
Then, their three-term recurrence becomes
\begin{equation}
\left(  1-\beta q^{n}\right)  xR_{n}\left(  x|\beta,q\right)  =R_{n+1}\left(
x|\beta,q\right)  +\left(  1-\beta^{2}q^{n-1}\right)  \left[  n\right]
_{q}R_{n-1}\left(  x|\beta,q\right)  \text{.} \label{R}%
\end{equation}
with $R_{-1}(x|\beta,q)\allowbreak=\allowbreak0$, $R_{0}(x|\beta
,q)\allowbreak=\allowbreak1$. Let us define the following density (following
\cite{KLS}(14.10.19) after necessary adjustments)
\begin{equation}
f_{C}(x|\beta,q)\allowbreak=\allowbreak\frac{(\beta^{2})_{\infty}}%
{(\beta,\beta q)_{\infty}}f_{h}\left(  x|q\right)  /\prod_{j=0}^{\infty
}l\left(  x|\beta q^{j}\right)  \text{,} \label{fC}%
\end{equation}
where $f_{h}$ is given by (\ref{fh}). Let us modify it by considering
\begin{equation}
f_{R}\left(  x|\beta,q\right)  =\sqrt{1-q}f_{C}(x\sqrt{1-q}/2|\beta
,q)/2\text{.} \label{fR}%
\end{equation}
Then we have the following orthogonal relationship satisfied by polynomials
$\left\{  R_{n}\right\}  $.
\begin{align}
&  \int_{S(q)}R_{n}\left(  x|\beta,q\right)  R_{m}\left(  x|\beta,q\right)
f_{R}\left(  x|\beta,q\right)  dx\label{intR^2}\\
&  =%
\begin{cases}
0\text{,} & \text{if }m\neq n\text{;}\\
\frac{\left[  n\right]  _{q}!(1-\beta)\left(  \beta^{2}\right)  _{n}}{\left(
1-\beta q^{n}\right)  }\text{,} & \text{if }m=n\text{.}%
\end{cases}
\nonumber
\end{align}

Polynomials $\left\{  R_{n}\right\}  $ are not monic. We can easily notice,
that the coefficient by $x^{n}$ in $R_{n}$ is $(\beta)_{n}$. Hence by defining
a new sequence of polynomials
\begin{equation}
V_{n}(x|\beta,q)\allowbreak=\allowbreak R_{n}(x|\beta,q)/(\beta)_{n}\text{,}
\label{monR}%
\end{equation}
we get the sequence of monic versions of the polynomials $\left\{
R_{n}\right\}  $.

Our main result concerns summing certain kernels built of polynomials
$\left\{  R_{n}\right\}  $. Therefore, we need the following lemma that
exposes the relationships between polynomials $\left\{  R_{n}\right\}  $ and
$\left\{  H_{n}\right\}  $.

\begin{lemma}
1)
\begin{align}
R_{n}(x|r,q)  &  =\sum_{k=0}^{\left\lfloor n/2\right\rfloor }\frac{\left[
n\right]  _{q}!}{\left[  k\right]  _{q}!\left[  n-2k\right]  _{q}!}%
q^{\binom{k}{2}}(-r)^{k}\left(  r\right)  _{n-k}H_{n-2k}(x|q)\text{,}%
\label{RnaH}\\
H_{n}(x|q)  &  =\sum_{k=0}^{\left\lfloor n/2\right\rfloor }\frac{\left[
n\right]  _{q}!(1-rq^{n-2k})}{\left[  k\right]  _{q}!\left[  n-2k\right]
_{q}!(1-r)(1-rq)_{n-k}}r^{k}R_{n-2k}(x|r,q)\text{.} \label{HnaR}%
\end{align}

\begin{align}
&  \int_{S(q)}H_{n}(x)R_{m}(x|\beta,q)f_{N}(x|q)dx=\label{HRfN}\\
&
\begin{cases}
0\text{,} & \text{if }n>m\text{ or }n+m\text{ is odd;}\\
q^{\binom{(m-n)/2}{2}}\frac{\left[  m\right]  _{q}!(-\beta)^{(m-n)/2}}{\left[
(m-n)/2\right]  _{q}!}\left(  \beta\right)  _{(m+n)/2}\text{,} &
\text{otherwise.}%
\end{cases}
\nonumber
\end{align}
and
\begin{gather}
\int_{S(q)}H_{n}(x)R_{m}(x|\beta,q)f_{R}(x|\beta,q)dx=\label{HRfR}\\%
\begin{cases}
0\text{,} & \text{if }m>n\text{ or }\left\vert n-m\right\vert \text{ is
odd;}\\
\frac{\beta^{(n-m)/2}\left(  \beta^{2}\right)  _{m}\left[  n\right]  _{q}%
!}{(1-\beta)\left[  (n-m)/2\right]  _{q}!\left(  \beta q\right)  _{(n+m)/2}%
}\text{,} & \text{ otherwise.}%
\end{cases}
\nonumber
\end{gather}

where density $f_{R}$ is given by (\ref{fR}) and moreover, can be presented in
one of the following equivalent forms:%
\begin{gather}
f_{R}(x|\beta,q)=(1-\beta)f_{CN}(x|x,\beta,q)=\label{fRR}\\
f_{N}(x|q)\frac{\left(  \beta^{2}\right)  _{\infty}}{\left(  \beta\right)
_{\infty}\left(  \beta q\right)  _{\infty}\prod_{j=0}^{\infty}\left(  (1+\beta
q^{j}\right)  ^{2}-(1-q)\beta q^{j}x^{2}))}=\nonumber\\
(1-\beta)f_{N}(x|q)\sum_{n\geq0}\frac{\beta^{n}H_{n}^{2}(x|q)}{\left[
n\right]  _{q}!}=(1-\beta)f_{N}(x|q)\sum_{n\geq0}\frac{\beta^{n}H_{2n}%
(x|q)}{\left[  n\right]  _{q}!\left(  \beta\right)  _{n+1}}\text{.}\nonumber
\end{gather}
We also have the following expansion of $f_{N}/f_{R}$ in orthogonal series in
polynomials $\left\{  R_{n}\right\}  :$%

\begin{equation}
f_{N}(x|q)=f_{R}(x|\gamma,q)\sum_{n\geq0}(-\gamma)^{n}q^{\binom{n}{2}}%
\frac{(\gamma)_{n}(1-\gamma q^{2n})}{\left[  n\right]  _{q}!(1-\gamma
)(\gamma^{2})_{2n}}R_{2n}(x|\gamma,q)\text{.} \label{fNfR}%
\end{equation}

2) We also have the following linearization formulae%
\begin{gather}
R_{n}(x|r,q)R_{m}(x|r,q)\label{RRnaR}\\
=\sum_{k=0}^{\min(n,m)}%
\genfrac{[}{]}{0pt}{}{m}{k}%
_{q}%
\genfrac{[}{]}{0pt}{}{n}{k}%
_{q}\left[  k\right]  _{q}!\frac{\left(  r\right)  _{m-k}\left(  r\right)
_{n-k}\left(  r\right)  _{k}\left(  r^{2}\right)  _{n+m-k}\left(
1-rq^{n+m-2k}\right)  }{\left(  1-r\right)  \left(  rq\right)  _{n+m-k}\left(
r^{2}\right)  _{m+n-2k}}\times\nonumber\\
R_{n+m-2k}(x|r,q)\text{,}\nonumber
\end{gather}

\begin{gather}
H_{m}(x|q)R_{n}(x|r,q)=\label{HRnaH}\\
\sum_{s=0}^{\left\lfloor (n+m)/2\right\rfloor }%
\genfrac{[}{]}{0pt}{}{n}{s}%
_{q}\left[  s\right]  _{q}!H_{n+m-2s}(x|q)\sum_{k=0}^{s}%
\genfrac{[}{]}{0pt}{}{m}{s-k}%
_{q}%
\genfrac{[}{]}{0pt}{}{n-s}{k}%
_{q}(-r)^{k}q^{\binom{k}{2}}(r)_{n-k}\text{,}\nonumber\\
H_{m}(x|q)R_{n}(x|r,q)=\label{HRnaR}\\
\sum_{u=0}^{\left\lfloor (n+m)/2\right\rfloor }\frac{\left[  n\right]
_{q}!\left[  m\right]  _{q}!(1-rq^{n+m-2u})}{\left[  u\right]  _{q}!\left[
n+m-2u\right]  _{q}!(1-r)}R_{n+m-2u}(x|r,q)\sum_{s=0}^{u}%
\genfrac{[}{]}{0pt}{}{u}{s}%
_{q}\frac{r^{u-s}}{\left(  rq\right)  _{n+m-u-s}}\times\nonumber\\
\sum_{k=0}^{s}%
\genfrac{[}{]}{0pt}{}{s}{k}%
_{q}%
\genfrac{[}{]}{0pt}{}{m+m-2s}{m+k-s}%
_{q}q^{\binom{k}{2}}(-r)^{k}\left(  r\right)  _{n-k}\text{.}\nonumber
\end{gather}

\end{lemma}

\begin{proof}
1) (\ref{HnaR}) and (\ref{RnaH}) are adaptations of (\ref{CnaC}) with either
$\beta\allowbreak=\allowbreak0$ or $\gamma\allowbreak=\allowbreak0$. When we
consider $\beta\allowbreak=\allowbreak0$, one has to be careful and notice,
that
\[
\beta^{n}(\gamma/\beta)_{n}=\prod_{j=0}^{n-1}(\beta-\gamma q^{j}%
)\rightarrow(-\gamma)^{n}q^{\binom{n}{2}}\text{,}%
\]
where the limit is taken when $\beta\rightarrow0$. When rescaling. to $S(q)$
we use formula (\ref{CnaR}). Formulae (\ref{HRfN}) and (\ref{HRfR}) follow
almost directly expansions respectively (\ref{RnaH}), (\ref{HnaR}) and the
fact that
\[
\int_{S(q)}H_{n}(x|q)H_{m}(x|q)f_{N}(x|q)dx=\left[  n\right]  _{q}!\delta
_{nm}\text{.}%
\]
in the first case and (\ref{intR^2}) in the second. Formula (\ref{fRR}) is
given in \cite{Szab19}(Proposition 1(3.3)). To get (\ref{fNfR}) we use
DEI(\ref{RnaH},\ref{HH}).

2) Again (\ref{RRnaR}) is an adaptation of the well-known formulae derived by
Rogers himself by the end of the nineteenth century concerning polynomial
$C_{n}$ related to polynomials $R_{n}$ by the formula (\ref{CnaR}). Formula
(\ref{HRnaH}) appeared in the version for polynomials $h$ and $C$ in
\cite{Szab2020}(8.3) but in fact, it was proved by Al-Salam and Ismail. in
\cite{ALIs88}. Formula (\ref{HRnaR}) is obtained directly by inserting
(\ref{HnaR}) into (\ref{HRnaH}).
\end{proof}

\subsection{Al-Salam--Chihara polynomials.}

Al-Salam--Chihara polynomials were defined first for $\left\vert x\right\vert
\leq1$ with $q$ and two other parameters $a$ and $b$ both from the segment
$[-1,1]$. We will consider $a$ and $b$ being complex conjugate. Let us define
the new parameters $\rho$ and $y$ in the following way: $ab\allowbreak
=\allowbreak\rho^{2}$ and $a+b\allowbreak=\allowbreak\frac{y}{\sqrt{1-q}}\rho$
and $y\in S\left(  q\right)  $. Then they will be denoted as $P_{n}%
(x|y,\rho,q)$ with these new parameters. The polynomials $\left\{
P_{n}\right\}  $, as demonstrated in \cite{Szab6}, can be interpreted in a
probabilistic manner as conditional expectations. The denotation
$P_{n}(x|y,\rho,q)$ reflects this conditional interpretation. It is known,
(see \cite{IA}, \cite{Szab-bAW} or \cite{Szab6}) that they satisfy the
following three-term recurrence:
\begin{equation}
P_{n+1}(x|y,\rho,q)=(x-\rho yq^{n})P_{n}(x|y,\rho,q)-(1-\rho^{2}%
q^{n-1})[n]_{q}P_{n-1}(x|y,\rho,q)\text{,} \label{3Pn}%
\end{equation}
with $P_{-1}(x|y,\rho,q)\allowbreak=\allowbreak0$ and $P_{0}(x|y,\rho
,q)\allowbreak=\allowbreak1$.

These polynomials are orthogonal with respect to the measure with the
following density:
\begin{equation}
f_{CN}\left(  x|y,\rho,q\right)  =f_{N}(x|q)\frac{(\rho^{2})_{\infty}%
}{W(x,y|\rho,q)}\text{,} \label{fCN}%
\end{equation}
where
\begin{equation}
W(x,y|\rho,q)\allowbreak=\allowbreak\prod_{k=0}^{\infty}\hat{w}_{q}\left(
x,y|\rho q^{k},q\right)  \allowbreak\text{,} \label{WW}%
\end{equation}
and
\begin{align*}
\hat{w}_{q}(x,y|\rho,q)\allowbreak &  =\allowbreak(1-\rho^{2})^{2}%
\allowbreak-\allowbreak(1-q)\rho xy(1+\rho^{2})\allowbreak+\allowbreak\rho
^{2}(1-q)(x^{2}+y^{2})\allowbreak\\
&  =\allowbreak w(x\sqrt{1-q}/2,y\sqrt{1-q}/2|\rho)\text{,}%
\end{align*}
where $w$ is given by (\ref{wxya}). M. Ismail. showed that
\begin{equation}
f_{CN}(x|y,\rho,q)\rightarrow\exp\left(  -\frac{(x-\rho y)^{2}}{2(1-\rho^{2}%
)}\right)  /\sqrt{2\pi(1-\rho^{2})}\text{,} \label{fCN1}%
\end{equation}
as $q\rightarrow1$. That is, $f_{CN}$ tends as $q->1^{-}$ to the density of
the normal $N(\rho y,1-\rho^{2})$ distribution. That is why $f_{CN}$ is called
conditional $q$-Normal.

The orthogonal relation for these polynomials are the following:
\begin{equation}
\int_{S\left(  q\right)  }P_{n}(x|y,\rho,q)P_{m}\left(  x|y,\rho,q\right)
f_{CN}\left(  x|y,\rho,q\right)  dx=%
\begin{cases}
0\text{,} & \text{if }m\neq n\text{;}\\
\left[  n\right]  _{q}!\left(  \rho^{2}\right)  _{n}\text{,} & \text{if
}m=n\text{.}%
\end{cases}
\label{PnPm}%
\end{equation}
Another fascinating property of the distribution $f_{CN}$ is the following
Chapman-Kolmogorov property:%
\begin{equation}
\int_{S\left(  q\right)  }f_{CN}\left(  z|y,\rho_{1},q\right)  f_{CN}\left(
y|x,\rho_{2},q\right)  dy=f_{CN}\left(  x|z,\rho_{1}\rho_{2},q\right)
\text{.} \label{ChK}%
\end{equation}

As shown in \cite{bms}, the relationship between the two families of
polynomials $\left\{  H_{n}\right\}  $ and $\left\{  P_{n}\right\}  $ are the
following:
\begin{align}
P_{n}\left(  x|y,\rho,q\right)   &  =\sum_{j=0}^{n}%
\genfrac{[}{]}{0pt}{}{n}{j}%
_{q}\rho^{n-j}B_{n-j}\left(  y|q\right)  H_{j}\left(  x|q\right)
\text{,}\label{PnaH}\\
H_{n}\left(  x|q\right)   &  =\sum_{j=0}^{n}%
\genfrac{[}{]}{0pt}{}{n}{j}%
_{q}\rho^{n-j}H_{n-j}\left(  y|q\right)  P_{j}\left(  x|y,\rho,q\right)
\text{,} \label{HnaP}%
\end{align}
where polynomials $\left\{  B_{n}\right\}  $ satisfy the following three-term
recurrence
\[
B_{n+1}(x|q)=-xq^{n}B_{n}(x|q)+q^{n-1}\left[  n\right]  _{q}B_{n-1}%
(x|q)\text{,}%
\]
with $B_{-1}(x|q)\allowbreak=\allowbreak0$, $B_{0}(x|q)\allowbreak
=\allowbreak1.$

It has been noticed in \cite{Szab-bAW} or \cite{Szab-rev} that the following
particular cases are true.

\begin{proposition}
\label{special}For $n\geq0$ we have

i) $R_{n}\left(  x|0,q\right)  \allowbreak=\allowbreak H_{n}\left(
x|q\right)  $,

ii) $R_{n}\left(  x|q,q\right)  \allowbreak=\allowbreak\left(  q\right)
_{n}U_{n}\left(  x\sqrt{1-q}/2\right)  $,

iii) $\lim_{\beta->1^{-}}\frac{R_{n}\left(  x|\beta,q\right)  }{\left(
\beta\right)  _{n}}\allowbreak=\allowbreak2\frac{T_{n}\left(  x\sqrt
{1-q}/2\right)  }{(1-q)^{n/2}}$,

iv) $R_{n}(x|\beta,1)\allowbreak=\allowbreak\left(  \frac{1+\beta}{1-\beta
}\right)  ^{n/2}H_{n}\left(  \sqrt{\frac{1-\beta}{1+\beta}}x\right)  $,

v) $R_{n}(x|\beta,0)\allowbreak=\allowbreak(1-\beta)U_{n}(x/2)-\beta
(1-\beta)U_{n-2}(x/2)$.

vi) $R_{n}(x|\beta,q)=P_{n}(x|x,\beta,q)$, where $\left\{  P_{n}\right\}  $
are defined by its three-term recurrence (\ref{3Pn}).
\end{proposition}

Having done all those preparations we are ready to present our main result.

\section{Positive, summable kernel built of $q$-ultraspherical polynomials
\label{main}}

One can find in the literature attempts to sum the kernels for $q$%
-ultraspherical polynomials like, e.g., \cite{GasRah} or \cite{Rahman97}.
However, the kernel we will present has a simple sum and moreover depends on
two parameters, more precisely on the symmetric function of two parameters,
not on only one.

Let us denote by $w_{n}(m,r_{1},r_{2},q)$ the following symmetric polynomial
of degree $2n$ and by $\phi_{n}(r_{1},r_{2},q)$ the following rational,
symmetric function, both in $r_{1}$ and $r_{2}$ that will often appear in the
sequel:%
\begin{align}
w_{n}(m,r_{1},r_{2},q)  &  =\sum_{s=0}^{n}%
\genfrac{[}{]}{0pt}{}{n}{s}%
_{q}r_{1}^{s}\left(  q^{m}r_{2}^{2}\right)  _{s}r_{2}^{n-s}\left(  q^{m}%
r_{1}^{2}\right)  _{n-s}\text{,}\label{wm}\\
\phi_{n}(r_{1},r_{2},q)  &  =\frac{w_{n}(0,r_{1},r_{2},q)}{\left(  r_{1}%
^{2}r_{2}^{2}\right)  _{n}}\text{.} \label{fi}%
\end{align}

Let us notice immediately, that
\[
w_{n}(m,r_{1},r_{2},q)=q^{-nm/2}w_{n}(0,r_{1}q^{m/2},r_{2}q^{m/2},q)\text{.}%
\]

\begin{theorem}
\label{glowny}The following symmetric bivariate kernel is nonnegative on
$S(q)\times S(q)$ and for $\left\vert r_{1}\right\vert ,\left\vert
r_{2}\right\vert <1:$%
\[
\sum_{n\geq0}\phi_{n}(r_{1},r_{2},q)\frac{(1-r_{1}r_{2}q^{n})}{\left[
n\right]  _{q}!(r_{1}^{2}r_{2}^{2})_{n}(1-r_{1}r_{2})}R_{n}(x|r_{1}%
r_{2},q)R_{n}(y|r_{1}r_{2},q)\text{,}%
\]
where $\left\{  R_{n}(x|\beta,q\right\}  _{n\geq0}$ are the $q$-ultraspherical
polynomials defined by the three-term recurrence (\ref{R}). Functions
$\left\{  w_{n}\right\}  $ are given by (\ref{wm}).

Moreover, we have%
\begin{gather}
f_{R}(x|r_{1}r_{2},q)f_{R}(y|r_{1}r_{2},q)\sum_{n\geq0}\frac{\phi_{n}%
(r_{1},r_{2},q)(1-r_{1}r_{2}q^{n})}{\left[  n\right]  _{q}!(r_{1}^{2}r_{2}%
^{2})_{n}(1-r_{1}r_{2})}R_{n}(x|r_{1}r_{2},q)R_{n}(y|r_{1}r_{2},q)
\label{LER}\\
=(1-r_{1}r_{2})f_{CN}\left(  y|x,r_{1},q\right)  f_{CN}\left(  x|y,r_{2}%
,q\right)  \text{.}\nonumber
\end{gather}
where $f_{CN}\left(  x|y,r_{2},q\right)  $ denotes the so-called conditional
$q$-Normal distribution, defined by (\ref{fCN}).

Denoting by $\hat{R}_{n}(x|r_{1}r_{2},q)\allowbreak=\allowbreak R_{n}%
(x|r_{1}r_{2},q)\sqrt{1-r_{1}r_{2}q^{n}}/\sqrt{\left[  n\right]  _{q}!\left(
r_{1}^{2}r_{2}^{2}\right)  _{n}(1-r_{1}r_{2})}$ the orthonormal version of the
polynomials $R_{n}$ we get more friendly version of our result%
\begin{align}
&  (1-r_{1}r_{2})f_{CN}\left(  y|x,r_{1},q\right)  f_{CN}\left(
x|y,r_{2},q\right) \label{LERO}\\
&  =f_{R}(x|r_{1}r_{2},q)f_{R}(y|r_{1}r_{2},q)\sum_{n\geq0}\phi_{n}%
(r_{1},r_{2},q)\hat{R}_{n}(x|r_{1}r_{2},q)\hat{R}_{n}(y|r_{1}r_{2}%
,q)\text{.}\nonumber
\end{align}

\end{theorem}

\begin{remark}
One of the referees reading this paper raised the question of convergence in
(\ref{LER}). Theorem \ref{exp2}, presented below, states that it is almost
uniform on $S(x)\times S(q)$. The main difficulty in proving Theorem
\ref{glowny} lies not in the convergence problems but in the transformation of
(\ref{exp1}) to (\ref{LER}). It is done by the series of operations like
changing the order of summation, introducing new variables and using
nontrivial identities, in other words, a very tedious, hard algebra.
\end{remark}

\begin{remark}
Recall that polynomials $R_{n}$ are closely connected with the classical
$q$-ultraspherical polynomial by the formula (\ref{CnaR}). So far two
successful summations of bivariate kernels built of $q$-ultraspherical
polynomials. In \cite{GasRah} and later generalized in \cite{Rahman97}
(formulae 1.7,1.8) the sum has a form (adopted to our case\}: $\sum_{n\geq
0}h_{n}C_{n}(x|r_{1}r_{2},q)C_{n}(y|r_{1}r_{2},q)t^{n}$ with a normalizing
sequence $\left\{  h_{n}\right\}  $. But by no means one can find such $t$ so
that $h_{n}t^{n}\allowbreak=\allowbreak\phi_{n}(r_{1},r_{2},q)$. The other,
independent summation of bivariate kernel related to $q$-ultraspherical
polynomials was done in \cite{Koelink99} (Theorem 3.3). Again, to adopt it to
the situation considered in Theorem \ref{glowny} we should take $c\allowbreak
=\allowbreak d\allowbreak=\allowbreak c^{\prime}\allowbreak=\allowbreak
d^{\prime}\allowbreak=\allowbreak0$ and $a\allowbreak=\allowbreak r_{1}%
r_{2}e^{i\theta}$, $b\allowbreak=\allowbreak r_{1}r_{2}e^{-i\theta}$
$a^{\prime}\allowbreak=\allowbreak r_{1}r_{2}e^{i\varphi}$ and $b^{\prime
}\allowbreak=\allowbreak r_{1}r_{2}e^{-i\varphi}$ with $x\allowbreak
=\allowbreak\cos\theta$ and $y\allowbreak=\allowbreak\cos\varphi$. This is so
since we have the assertion vi) of Proposition \ref{special}. But then again
since in this case sequence $\left\{  H_{n}\right\}  $ reduces to
$C(x,y)/\left(  q,r_{1}r_{2}\right)  _{n}$ and we cannot find $t$ such that
$t^{n}/\left(  q,r_{1}r_{2}\right)  _{n}\allowbreak=\allowbreak\phi_{n}%
(r_{1},r_{2},q)$ for all $n$.

Hence, we deduce that our result is aside known results and is completely new.
\end{remark}

\begin{remark}
Now we can apply our results and complement the results of \cite{Szab19}. Let
us recall, that in this paper the following $3-$dimensional distribution
\[
f_{3D}(x,y,z|\rho_{12},\rho_{13},\rho_{23},q)=(1-r)f_{CN}(x|y,\rho
_{12},q)f_{CN}(y|z,\rho_{23},q)f_{CN}(z|x,\rho_{13},q),
\]
where we denoted $r\allowbreak=\allowbreak\rho_{12}\rho_{23}\rho_{13}$, has
the property that all its conditional moments are the polynomials in the
conditioning random variable(s). Hence, it is a compactly supported
generalization of the $3$-dimensional Normal distribution. Let us recall also,
that the one-dimensional marginals are the same and equal to $f_{R}(.|r,q)$.
Moreover, the two-dimensional marginal distribution of $(Y,Z)$ is equal to
$(1-r)f_{CN}(y|z,\rho_{23},q)f_{CN}(z|y,\rho_{12}\rho_{13},q)$ and similarly
for the other two two-dimensional marginals. Consequently, we can now expand
it in the following way:
\begin{align}
&  (1-r)f_{CN}(y|z,\rho_{23},q)f_{CN}(z|y,\rho_{12}\rho_{13},q)\label{LE_R}\\
&  =f_{R}(z|r,q)f_{R}(y|r,q)\sum_{n\geq0}\phi_{n}(\rho_{23},\rho_{13}\rho
_{12},q)\hat{R}_{n}(z|r,q)\hat{R}_{n}(y|r,q).\nonumber
\end{align}
As a result we can simplify the formula for the conditional moment $E(.|Z)$.
Namely, we have%
\begin{equation}
E(\hat{R}_{n}(Y|r,q)|Z)\allowbreak=\allowbreak\phi_{n}(\rho_{23},\rho_{13}%
\rho_{12},q)\hat{R}_{n}(Z|r,q), \label{ER}%
\end{equation}
as it follows from our main result.

Now recall that all marginal distributions are the same and have a density
$f_{R}$. Further, recall that all conditional moments are polynomials in the
conditioning random variable of order not exceeding the order of the moment.
Hence, we could deduce from the main result of \cite{Szabl21} that there
should exist a LE of the joint two-dimensional marginal involving polynomials
$R_{n}$ as the ones being orthogonal with respect to the one-dimensional
marginal distribution. So (\ref{LE_R}) presents this lacking LE of the
two-dimensional distribution.
\end{remark}

\begin{remark}
As another application of our result, one could think of constructing a
stationary Markov process with marginals $f_{R}$ and transitional density
given by (\ref{LE_R}). It would be the first application of $q$-ultraspherical
polynomials in the theory of Stochastic processes.
\end{remark}

\begin{remark}
Finally, we have the following two theoretical results related one to another.
Namely, true is the following expansion
\begin{align*}
&  f_{R}(x|r_{1}r_{2},q)f_{R}(y|r_{1}r_{2},q)\sum_{n\geq0}\phi_{n}(r_{1}%
,r_{2},q)\hat{R}_{n}(x|r_{1}r_{2},q)\hat{R}_{n}(y|r_{1}r_{2},q)\\
&  =(1-r_{1}r_{2})f_{N}(x|q)f_{N}(y|q)\sum_{n,m\geq0}\frac{r_{1}^{n}r_{2}^{m}%
}{\left[  n\right]  _{q}!\left[  m\right]  _{q}!}H_{n}(x|q)H_{m}%
(x|q)H_{n}(y|q)H_{m}(y|q),
\end{align*}
where $H_{n}$ are the $q$-Hermite polynomials defined by the three-term
recurrence (\ref{3trH}) and $f_{N}$ is the $q$-Normal density defined by
(\ref{fN}). The convergence, as it follows from Rademacher-Menshov Theorem is
for almost all $x$ and $y$ from $S(q)$, provided of course if $\left\vert
r_{1}\right\vert ,\left\vert r_{2}\right\vert <1$.

Further, after cancelling out $f_{N}(x|q)f_{N}(y|q)$ on both its sides and
using definition of $f_{CN}$ given by (\ref{fCN}), the Poisson-Mehler formula
(\ref{PMf}), expansion (\ref{exp13}) with $m\allowbreak=\allowbreak0$ and of
course (\ref{LERO}). As far as the convergence is concerned we apply the
well-known Rademacher-Menshov Theorem and use the fact that product measure
$f_{N}(x|q)\times f_{N}(y|q)$ is absolutely continuous with respect to
Lebesgue measure on $S(q)\times S(q)$,%
\begin{gather*}
(1-r_{1}r_{2})\left(  \sum_{j\geq0}\frac{r_{1}^{j}}{\left[  j\right]  _{q}%
!}H_{j}(x|q)H_{j}(y|q)\right)  \left(  \sum_{j\geq0}\frac{r_{2}^{j}}{\left[
j\right]  _{q}!}H_{j}(x|q)H_{j}(y|q)\right)  =\left(  \sum_{k\geq0}%
\frac{(r_{1}r_{2})^{k}}{\left[  k\right]  _{q}!(r_{1}^{2}r_{2}^{2})_{k}}%
H_{2k}(x|q)\right) \\
\times\left(  \sum_{k\geq0}\frac{(r_{1}r_{2})^{k}}{\left[  k\right]
_{q}!(r_{1}^{2}r_{2}^{2})_{k}}H_{2k}(y|q)\right)  \sum_{n\geq0}\frac{\phi
_{n}(r_{1},r_{2},q)(1-r_{1}r_{2}q^{n})}{\left[  n\right]  _{q}!(r_{1}^{2}%
r_{2}^{2})_{n}(1-r_{1}r_{2})}R_{n}(x|r_{1}r_{2},q)R_{n}(y|r_{1}r_{2},q),
\end{gather*}
for all $x,y\in S(q)$, $\left\vert r_{1}\right\vert ,\left\vert r_{2}%
\right\vert ,\left\vert q\right\vert <1$. The convergence is almost everywhere
on $S\left(  q\right)  \times S\left(  q\right)  $ with respect to the product
Lebesgue measure.
\end{remark}

\begin{corollary}
i) Setting $r_{1}\allowbreak=\allowbreak\rho$ and $r_{2}\allowbreak
=\allowbreak0$ we have the following expansion which is true for all $x,y\in
S(q),\left\vert \rho\right\vert <1$, $-1<q\leq1$.
\begin{equation}
f_{N}(x|q)f_{N}(y|q)\sum_{n\geq0}\frac{\rho^{n}}{\left[  n\right]  _{q}!}%
H_{n}(x|q)H_{n}(y|q)=f_{N}(y|q)f_{CN}(x|y,\rho,q). \label{PMf}%
\end{equation}
In other words, the well-known Poisson-Mehler expansion formula is a
particular case of (\ref{LER}).

ii) Let us set $r_{1}\allowbreak=\allowbreak r\allowbreak=\allowbreak-r_{2}$.
We have then
\[
\phi_{n}(r,-r,q)=\left\{
\begin{array}
[c]{ccc}%
0 & \text{if} & n\text{ is odd}\\
r^{2k}\left(  q|q^{2}\right)  _{k}/\left(  r^{4}q|q^{2}\right)  _{k} &
\text{if} & n=2k
\end{array}
\right.  ,
\]
and consequently we get%
\begin{gather*}
(1+r^{2})f_{CN}\left(  y|x,r,q\right)  f_{CN}\left(  x|y,-r,q\right)  =\\
f_{R}(x|-r^{2},q)f_{R}(y|-r^{2},q)\sum_{k\geq0}r^{2k}\left(  q|q^{2}\right)
_{k}/\left(  r^{4}q|q^{2}\right)  _{k}\hat{R}_{2k}(x|-r^{2},q)\hat{R}%
_{2k}(y|-r^{2},q).
\end{gather*}

iii) Taking $q\allowbreak=\allowbreak1$, we get:
\begin{gather*}
\frac{1+r_{1}r_{2}}{2\pi(1-r_{1}r_{2})}\exp(-\frac{(1-r_{1}r_{2})x^{2}%
}{2(1+r_{1}r_{2})}-\frac{(1-r_{1}r_{2})y^{2}}{2(1+r_{1}r_{2})})\times\\
\sum_{n\geq0}\left(  \frac{r_{1}+r_{2}}{\left(  1+r_{1}r_{2}\right)
(1-r_{1}^{2}r_{2}^{2})}\right)  ^{n}\left(  \frac{1+r_{1}r_{2}}{1-r_{1}r_{2}%
}\right)  ^{n}H_{n}\left(  x\sqrt{\frac{1-r_{1}r_{2}}{1+r_{1}r_{2}}}\right)
H_{n}\left(  y\sqrt{\frac{1-r_{1}r_{2}}{1+r_{1}r_{2}}}\right) \\
=(1-r_{1}r_{2})\exp\left(  -\frac{\left(  x-r_{1}y\right)  ^{2}}{2\left(
1-r_{1}\right)  ^{2}}--\frac{\left(  y-r_{2}x\right)  ^{2}}{2\left(
1-r_{2}\right)  ^{2}}\right)  /(2\pi(1-r_{1})(1-r_{2})).
\end{gather*}

\end{corollary}

\begin{proof}
i) We use the fact that $R_{n}(x|0,q)\allowbreak=\allowbreak H_{n}(x|q)$ and
$f_{N}(x|0,q)\allowbreak=\allowbreak f_{N}(x|q)$ and also that $\sum_{j=0}^{n}%
\genfrac{[}{]}{0pt}{}{n}{j}%
_{q}r_{1}^{j}\left(  r_{2}^{2}\right)  _{j}r_{2}^{n-j}\left(  r_{1}%
^{2}\right)  _{n-j}$ $\allowbreak=\allowbreak r_{1}^{n}$ when $r_{2}%
\allowbreak=\allowbreak0$.

ii) We have%
\begin{gather*}
\phi_{n}(r,-r,q)=\frac{1}{\left(  -r^{2}\right)  _{n}}r^{n}\sum_{j=0}%
^{n}(-1)^{j}(r^{2})_{j}\left(  r^{2}\right)  _{n-j}\\
=\left\{
\begin{array}
[c]{ccc}%
0 & \text{if} & n\text{ is odd}\\
r^{2k}\left(  q|q^{2}\right)  _{k}\left(  r^{2}|q^{2}\right)  _{k}/\left(
-r^{2}\right)  _{2k} & \text{if} & n=2k
\end{array}
\right.  .
\end{gather*}

iii) We use (\ref{fCN1}) with $x\allowbreak=\allowbreak y$, Proposition
(\ref{special}), the fact that in this case $\phi_{n}(r_{1},r_{2}%
,1)\allowbreak=\allowbreak(\frac{r_{1}+r_{2}}{1+r_{1}r_{2}})^{n}$ and the fact
that polynomials $H_{n}(\sqrt{a}x)$ are orthogonal with respect to the measure
with the density$\exp(-\alpha x^{2}/2)/\sqrt{2\pi\alpha}$.
\end{proof}

Before we present a complicated proof of this theorem, let us formulate and
prove some auxiliary results.

\begin{theorem}
\label{exp2}For all $x,y\in S(q)$, $r_{1},r_{2}\in(-1,1)$ and $-1<q\leq1$ we
have
\begin{align}
0  &  \leq(1-r_{1}r_{2})f_{CN}\left(  y|x,r_{1},q\right)  f_{CN}\left(
x|y,r_{2},q\right) \label{exp1}\\
&  =f_{N}(x|q)f_{R}(y|r_{1}r_{2},q)\sum_{n\geq0}\frac{1}{\left[  n\right]
_{q}!}H_{n}(x|q)D_{n}(y|r_{1},r_{2},q),\nonumber
\end{align}
where
\begin{equation}
D_{n}(y|r_{1},r_{2},q)=\sum_{j=0}^{n}%
\genfrac{[}{]}{0pt}{}{n}{j}%
_{q}r_{1}^{n-j}r_{2}^{j}\left(  r_{1}\right)  _{j}H_{n-j}(y|q)R_{j}%
(y|r_{1}r_{2},q)/\left(  r_{1}^{2}r_{2}^{2}\right)  _{j}. \label{Dn}%
\end{equation}
and the convergence is absolute and almost uniform on $S(q)\times S(q)$. We
also have%
\begin{equation}
\int_{S(q)}(1-r_{1}r_{2})f_{CN}\left(  y|x,r_{1},q\right)  f_{CN}\left(
x|y,r_{2},q\right)  dx=f_{R}(y|r_{1}r_{2},q). \label{mar}%
\end{equation}

\end{theorem}

\begin{proof}
This Theorem is composed of the results that appeared in \cite{Szab6} and,
e.g., \cite{Szab19}. Namely, in \cite{Szab6} the following result was proved
(Thm.3 (3.4))%
\begin{align}
&  \allowbreak\frac{f_{CN}\left(  z|x,\rho_{2},q\right)  f_{CN}\left(
x|y,\rho_{1},q\right)  f_{N}\left(  y|q\right)  }{f_{CN}\left(  z|y,\rho
_{1}\rho_{2},q\right)  f_{N}\left(  y|q\right)  }\label{expfc2n}\\
&  =f_{N}\left(  x|q\right)  \sum_{j=0}^{\infty}\frac{1}{\left[  j\right]
_{q}!}H_{j}\left(  x|q\right)  C_{j}\left(  y,z|\rho_{1},\rho_{2},q\right)
,\nonumber
\end{align}
where
\[
C_{n}\left(  y,z|\rho_{1},\rho_{2},q\right)  =\sum_{s=0}^{n}\allowbreak%
\genfrac{[}{]}{0pt}{}{n}{s}%
_{q}\rho_{1}^{n-s}\rho_{2}^{s}\left(  \rho_{1}^{2}\right)  _{s}H_{n-s}\left(
y|q\right)  P_{s}\left(  z|y,\rho_{1}\rho_{2},q\right)  /(\rho_{1}^{2}\rho
_{2}^{2})_{s}.
\]
convergence in (\ref{expfc2n}) is absolute and almost uniform on $S(q)\times
S(x)\times S(q)$. Now in \cite{Szab19} has been noticed (Proposition 1(3.1))
that $(1-r)f_{CN}(y|y,r,q)\allowbreak=\allowbreak f_{R}(y|r,q)$ where $f_{R}$
is the density of the measure that makes polynomials $\left\{  R_{n}\right\}
$ orthogonal. Thus, by replacing $\rho_{1}$ and $\rho_{2}$ by $r_{1}$ and
$r_{2}$ and identifying $y$ and $z$ we get (\ref{exp1}). Let us denote now
\[
D_{n}(y|\rho_{1},\rho_{2},q)=C_{n}(y,y|\rho_{1},\rho_{2},q),
\]

\end{proof}

\begin{remark}
\label{sym}$(1-r_{1}r_{2})f_{CN}\left(  y|x,r_{1},q\right)  f_{CN}\left(
x|y,r_{2},q\right)  $ is a symmetric function with respect both $x$ and $y$ as
well as $r_{1}$ and $r_{2}$.

To see this, let us refer to the definition of the density $f_{CN}$. Namely,
we have%
\begin{align*}
&  (1-r_{1}r_{2})f_{CN}\left(  y|x,r_{1},q\right)  f_{CN}\left(
x|y,r_{2},q\right) \\
&  =(1-r_{1}r_{2})\left(  r_{1}^{2},r_{2}^{2}\right)  _{\infty}f_{N}%
(x|q)f_{N}(y|q)/(W(x,y|r_{1},q)W(x,y|r_{2},q))
\end{align*}
where $W(x,y|r_{1},q)$ is defined by (\ref{WW}).
\end{remark}

The next several partial results require very special and tedious calculations
that can be interested only for those who are working on $q$-series theory. To
preserve the logic of the arguments leading to the result, we will move all
such auxiliary results to the last section.

Our following result presents the functions $\left\{  D_{n}\right\}  $ given
by (\ref{Dn}) as a combination of only polynomials $\left\{  R_{n}\right\}  $
or in other words expands each $D_{n}$ in the basis of $\left\{
R_{n}\right\}  $. The proof is simple ideologically, but very hard in terms of
specialized calculations, which is why we will shift it to the last section.

\begin{proposition}
\label{rozw}We have for all $n\geq0$, $r_{1},r_{2}\in(-1,1)$ and $-1<q\leq1$:%
\begin{equation}
D_{n}(y|r_{1},r_{2},q)=\sum_{u=0}^{\left\lfloor n/2\right\rfloor }%
\frac{\left[  n\right]  _{q}!(1-r_{1}r_{2}q^{n-2u})(r_{1}r_{2})^{u}}{\left[
u\right]  _{q}!\left[  n-2u\right]  _{q}!\left(  r_{1}r_{2}\right)  _{n-u+1}%
}R_{n-2u}(y|r_{1}r_{2},q)\gamma_{n,u}(r_{1},r_{2},q), \label{Dnn}%
\end{equation}
where
\begin{align}
\gamma_{n,u}(r_{1},r_{2},q)  &  =\frac{1}{(r_{1}^{2}r_{2}^{2})_{2n}}\sum
_{m=0}^{u}%
\genfrac{[}{]}{0pt}{}{u}{m}%
_{q}\left(  r_{2}^{2}\right)  _{m}\left(  r_{1}r_{2}q^{n-u-m+1}\right)
_{m}\label{fin}\\
&  \times\left(  r_{1}^{2}\right)  _{m}\left(  r_{1}r_{2}\right)  _{m}%
w_{n-2m}(m,r_{1}r_{2},q),\nonumber
\end{align}
is well defined for all $n\geq0$ and $u\leq0\leq\left\lfloor n/2\right\rfloor
$.
\end{proposition}

\begin{proof}
The long, tedious proof is shifted to Section \ref{aux1}.
\end{proof}

Before we formulate the next Lemma, let us formulate the following simple result.

\begin{proposition}
\label{aux3}For all $x\in S(x)$, $\left\vert r\right\vert <1$, $\left\vert
q\right\vert \leq1$, $m\geq0$ we have%
\begin{gather}
(1-r)f_{N}(x|q)\sum_{j\geq0}H_{j}(x|q)H_{j+m}(x|q)r^{j}/\left[  j\right]
_{q}!=f_{R}(x|q,q)R_{m}(x|r,q)/\left(  r^{2}\right)  _{m}=\label{exp11}\\
f_{R}(x|r,q)\sum_{k=0}^{m}%
\genfrac{[}{]}{0pt}{}{m}{k}%
_{q}(-r)^{k}q^{\binom{k}{2}}H_{m-j}(x|q)R_{j}(x|r,q)/\left(  r^{2}\right)
_{j}=\label{exp12}\\
(1-r)f_{N}(x|q)\sum_{s\geq0}\frac{r^{s}}{\left[  s\right]  _{q}!(r)_{m+s+1}%
}H_{2s+m}(x|q). \label{exp13}%
\end{gather}

\end{proposition}

\begin{proof}
We start with the result of \cite{Szab6}(Lemma 3i)) where we set
$x\allowbreak=\allowbreak y$ and make other adjustments and get directly
(\ref{exp11}) and (\ref{exp12}).

To get (\ref{exp13}) we must apply the linearization formula for the
$q$-Hermite polynomials and (\ref{binT}):%
\begin{gather*}
\sum_{j\geq0}\frac{r^{j}}{\left[  j\right]  _{q}!}H_{j}(x|q)H_{j+m}(x|q)\\
=\sum_{j\geq0}\frac{r^{j}}{\left[  j\right]  _{q}!}\sum_{k=0}^{j}%
\genfrac{[}{]}{0pt}{}{j}{k}%
_{q}%
\genfrac{[}{]}{0pt}{}{j+m}{k}%
_{q}\left[  k\right]  _{q}!H_{2j+m-2k}(x|q)\\
=\sum_{j\geq0}\frac{r^{j}}{\left[  j\right]  _{q}!}\sum_{s=0}^{j}%
\genfrac{[}{]}{0pt}{}{j}{s}%
_{q}%
\genfrac{[}{]}{0pt}{}{j+m}{m+s}%
_{q}\left[  j-s\right]  _{q}!H_{2s+m}(x|q)\\
=\sum_{s\geq0}\frac{r^{s}}{\left[  s\right]  _{q}!}H_{2s+m}(x|q)\sum_{j\geq
s}r^{j-s}%
\genfrac{[}{]}{0pt}{}{j-s+(m+s)}{m+s}%
_{q}\\
=\sum_{s\geq0}\frac{r^{s}}{\left[  s\right]  _{q}!(r)_{m+s+1}}H_{2s+m}(x|q).
\end{gather*}

\end{proof}

\begin{lemma}
\label{final}We have $n\geq0:$
\begin{gather*}
\int_{S(q)}R_{n}(x|r_{1}r_{2},q)(1-r_{1}r_{2})f_{CN}\left(  y|x,r_{1}%
,q\right)  f_{CN}\left(  x|y,r_{2},q\right)  dx\\
=\phi_{n}(r_{1},r_{2},q)R_{n}(y|r_{1}r_{2},q).
\end{gather*}
and similarly for the integral with respect to $y$.

Further, for all $n\geq0$ and $k\leq\left\lfloor n/2\right\rfloor $
\[
\gamma_{n,k}(r_{1},r_{2},q)=\phi_{n-2k}(r_{1},r_{2},q).\allowbreak
\]
In particular we have
\begin{align*}
\gamma_{2k,k}(r_{1},r_{2},q)  &  =1,\\
\gamma_{n,0}(r_{1},r_{2},q)  &  =\phi_{n}(r_{1},r_{2},q).
\end{align*}

\end{lemma}

\begin{proof}
The long, detailed proof by induction is shifted to Section \ref{aux1}.
\end{proof}

\begin{proof}
[Proof of the Theorem \ref{glowny}]Now, it is enough to recall results of say
\cite{Szabl21} to get our expansion.
\end{proof}

As a corollary, we get the following nice identities that seem to be of
interest by themselves.

\begin{corollary}
i) For all $r_{1},r_{2},q\in\mathbb{C}$, $n\geq0$, $0\leq u\leq\left\lfloor
n/2\right\rfloor $ we have
\begin{gather*}
\sum_{m=0}^{u}%
\genfrac{[}{]}{0pt}{}{u}{m}%
_{q}\left(  r_{2}^{2}\right)  _{m}\left(  r_{1}r_{2}q^{n-u-m+1}\right)
_{m}\left(  r_{1}^{2}\right)  _{m}\left(  r_{1}r_{2}\right)  _{m}%
w_{n-2m}(m,r_{1},r_{2},q)\\
=\left(  r_{1}^{2}r_{2}^{2}q^{n-2u}\right)  _{2u}w_{n-2u}(0,r_{1},r_{2},q),
\end{gather*}
since we have polynomials on both sides of the equations.

ii) For all $\left\vert r_{1}\right\vert ,\left\vert r_{2}\right\vert
<1,q\in(-1,1)$, $n\geq0$%
\begin{align*}
\left\vert \phi_{n}(r_{1},r_{2},q)\right\vert  &  \leq1,\\
\sum_{n\geq0}\left\vert \phi_{n}(r_{1},r_{2},q)\right\vert ^{2}  &  <\infty.
\end{align*}

\end{corollary}

\begin{proof}
i) The proof follows directly the definition of $\phi_{n}$ and its proved
properties. ii) The fact that $\sum_{n\geq0}\left\vert \phi_{n}(r_{1}%
,r_{2},q)\right\vert ^{2}<\infty$, follows directly the fact that we are
dealing with the mean-square expansion. The fact that $\left\vert \phi
_{n}(r_{1},r_{2},q)\right\vert \leq1$, follows the probabilistic
interpretation of our result. Namely recall, (\ref{ER}) with $r_{1}%
\allowbreak=\allowbreak\rho_{23}$, $r_{2}\allowbreak=\allowbreak\rho_{13}%
\rho_{12}$ and the fact that $E\hat{R}_{n}(Y|r_{1}r_{2},q)\allowbreak
=\allowbreak0$ and $E\hat{R}_{n}^{2}(Y|r_{1}r_{2},q)\allowbreak=\allowbreak1$.
Now recall, the well-known fact that the variance of the conditional
expectation of a random variable, does not exceed its variance. Consequently,
we have
\begin{gather*}
1=E\hat{R}_{n}^{2}(r_{1}r_{2},q)\geq E(E(\hat{R}_{n}(Y|r,q)|Z))^{2}=\\
\phi_{n}(r_{1},r_{2},q)^{2}E\hat{R}_{n}^{2}(r_{1}r_{2},q)=\phi_{n}(r_{1}%
,r_{2},q)^{2}.
\end{gather*}

\end{proof}

\section{List of Lancaster kernels\label{lista}}

\subsection{Symmetric kernels}

Below, we list symmetric Lancaster kernels that can be easily obtained using,
mentioned in the Introduction, the idea of expansion of the ratio of two
densities. The list below has simple sums and sometimes leads to smooth
stationary Markov processes.

1. We start with the following LK called the Poisson-Mehler kernel.
\begin{equation}
f_{CN}(x|y,\rho,q)/f_{N}(x|q)=\sum_{n\geq0}\frac{\rho^{n}}{\left[  n\right]
_{q}!}H_{n}(x|q)H_{n}(y|q). \label{P-M}%
\end{equation}

Its justification is DEI(\ref{HnaP},\ref{anH}). It leads to the so-called
$q$-Ornstein-Uhlenbeck process, for details see, e.g., \cite{Szab-OU-W}.

2. One should mention the following particular case of the above-mentioned
formula, that is $q\mathbb{\allowbreak=\allowbreak}0$. Then $H_{n}%
(x|0)\allowbreak=\allowbreak U_{n}(x/2)$ and $\left[  n\right]  _{0}%
!\allowbreak=\allowbreak1$ and finally we get for all $\left\vert x\right\vert
,\left\vert y\right\vert \leq2,\left\vert \rho\right\vert <1$:%
\begin{equation}
\sum_{n\geq0}\rho^{n}U_{n}(x/2)U_{n}(y/2)\allowbreak=\allowbreak\frac
{1-\rho^{2}}{(1-\rho^{2})^{2}-\rho(1+\rho^{2})xy+\rho^{2}(x^{2}+y^{2})}.
\label{UU}%
\end{equation}

Recently, in \cite{SzabCheb} (\ref{UU}) formula has been proven by other means
together with the following one:

3.
\[
\sum_{n\geq0}\rho^{n}T_{n}(x/2)T_{n}(y/2)=\frac{4(1-\rho^{2})-\rho(3+\rho
^{2})xy+2\rho^{2}(x^{2}+y^{2})}{4((1-\rho^{2})^{2}-\rho(1+\rho^{2})xy+\rho
^{2}(x^{2}+y^{2}))}.
\]

4. The following expansion appeared recently in \cite{Szab22}(4.7):
\begin{gather*}
1+2\sum_{n\geq1}\rho^{n^{2}}T_{n}(x)T_{n}(y)=\\
\frac{1}{2}\left(  \theta(\rho,\arccos(x)-\arccos(y)\right)  +\theta
(\rho,\arccos(x)+\arccos(y)),
\end{gather*}
where $\theta$ denotes the Jacobi Theta function.

5. The following expansion appeared recently in \cite{Szab22}(unnamed
formula):
\begin{gather*}
\frac{4}{\pi^{2}}\sqrt{(1-x^{2})(1-y^{2})}\sum_{n\geq0}\rho^{n(n+2)}%
U_{n}(x)U_{n}(y)=\\
\frac{1}{\rho\pi^{2}}\left(  \theta(\rho,\arccos(x)-\arccos(y)\right)
-\theta(\rho,\arccos(x)+\arccos(y)),
\end{gather*}
where, as before, $\theta$ denotes Jacobi Theta function.

6. It is well-known that the following LK is also true
\begin{gather*}
(1-\rho)\sum_{n\geq0}\frac{n!}{\Gamma(n+\alpha)}L_{n}^{\alpha}(x)L_{n}%
^{\alpha}(y)\rho^{n}=\\
\exp\left(  -\rho\frac{x+y}{1-\rho}\right)  I_{\alpha-1}\left(  \frac
{2(x,y\rho)^{1/2}}{1-\rho}\right)  /(xy\rho)^{(\alpha-1)/2},
\end{gather*}
where $L_{n}^{(\alpha)}(x)$ denotes generalized Laguerre polynomials i.e. the
ones orthogonal with respect to the measure with the density $\exp
(-x)x^{\alpha-1}/\Gamma(\alpha)$ for $x\geq0$ and $\alpha\geq1$. $I_{\alpha
-1}(x)$ denotes a modified Bessel function of the first kind. This kernel also
leads to a smooth, stationary Markov process, as shown in \cite{Szab22}.

Let us remark that, following \cite{Szab22}, the kernels mentioned at points
1., 4., 5., 6. allow to generate stationary Markov processes with polynomial
conditional moments that allow continuous path modification.

\subsection{Non-symmetric kernels}

It has to be underlined that the list of non-symmetric kernels presented below
is far from being complete. There is nice paper, namely \cite{suslov96}
mentioning more of them. However, they are in a very complicated form,
interesting only to specialists in $q$-series theory.

1. We start with the following, known, but recently presented with general
setting in \cite{SzabCheb}:%
\[
\sum_{n\geq0}\rho^{n}U_{n}(x/2)T_{n}(y/2)=\frac{2(1-\rho^{2})+\rho^{2}%
x^{2}-\rho xy}{2((1-\rho^{2})^{2}-\rho(1+\rho^{2})xy+\rho^{2}(x^{2}+y^{2}))}.
\]

2. In \cite{Szab-bAW}(3.12) we have the kernel that after adjusting it to
ranges of $x,y,z$ equal to $S\left(  q\right)  $ and utilizing and the fact
that
\begin{gather*}
f_{CN}(y|x,\rho_{1},q)f_{CN}(x|z,\rho_{2},q)/(f_{CN}(y|z,\rho_{1}\rho
_{2},q)f_{CN}(x|y,\rho_{1},q))\allowbreak\\
=\allowbreak f_{CN}(z|y,\rho_{2},q)/f_{CN}(z|x,\rho_{2},q),
\end{gather*}
we get finally the following non-symmetric kernel:
\[
\sum_{j\geq0}\frac{\rho_{2}^{j}}{\left[  j\right]  _{q}!\left(  \rho_{1}%
^{2}\rho_{2}^{2}\right)  _{j}}P_{j}\left(  z|y,\rho_{1}\rho_{2},q\right)
P_{j}\left(  x|y,\rho_{1},q\right)  =\frac{f_{CN}\left(  z|x,\rho
_{2},q\right)  }{f_{CN}(z|y,\rho_{1}\rho_{2},q)},
\]
true for all $\left\vert x\right\vert ,\left\vert y\right\vert ,\left\vert
z\right\vert \in S(q),\left\vert \rho_{1}\right\vert ,\left\vert \rho
_{1}\right\vert <1$ and $\left\vert q\right\vert \leq1$, where $P_{j}%
(x|y,\rho,q)$ and $f_{CN}(x|y,\rho,q)$ are defined respectively by (\ref{3Pn})
and (\ref{fCN}).

2s. Using the fact that
\[
P_{n}(x|y,\rho,1)\allowbreak=\allowbreak(1-\rho^{2})^{n/2}H_{n}\left(  (x-\rho
y)/\sqrt{1-\rho^{2}}\right)  ,
\]
and%
\[
f_{CN}(x|y,\rho,1)\allowbreak=\frac{1}{\sqrt{2\pi(1-\rho^{2})}}\exp\left(
-\frac{\left(  x-\rho y\right)  ^{2}}{2(1-\rho^{2})}\right)  ,
\]
$\allowbreak$by \cite{Szab2020}(8.24) and \cite{Szab2020}(8.32) we get the
following kernel
\begin{align*}
&  \sum_{j\geq0}\frac{\rho_{2}^{j}(1-\rho_{1}^{2})^{j/2}}{j!(1-\rho_{1}%
^{2}\rho_{2}^{2})^{j/2}}H_{j}(\frac{z-\rho_{1}\rho_{2}y}{\sqrt{1-\rho_{1}%
^{2}\rho_{2}^{2}}}H_{j}(\frac{x-\rho_{1}y}{\sqrt{1-\rho_{1}^{2}}})\\
&  =\sqrt{\frac{1-\rho_{1}^{2}\rho_{2}^{2}}{1-\rho_{2}^{2}}}\exp\left(
\frac{(z-\rho_{1}\rho_{2}y)^{2}}{2(1-\rho_{1}^{2}\rho_{2}^{2})}-\frac
{(z-\rho_{2}x)^{2}}{2(1-\rho_{2}^{2})}\right)  .
\end{align*}

3. The following non-symmetric kernel was presented in \cite{SzablKer}. For
$\left\vert b\right\vert >\left\vert a\right\vert ,\left\vert q\right\vert
<1$, $x,y\in S\left(  q\right)  $%
\[
\sum_{n\geq0}\frac{a^{n}}{\left[  n\right]  _{q}!b^{n}}H_{n}(x|a,q)H_{n}%
\left(  y|b,q\right)  \allowbreak=\allowbreak\left(  a^{2}/b^{2}\right)
_{\infty}\prod_{k=0}^{\infty}\frac{\left(  1-(1-q)xaq^{k}+(1-q)a^{2}%
q^{2k}\right)  }{w\left(  x,y,q^{k}a/b|q\right)  },
\]
where $H_{n}(x|a,q)$ denotes the so-called big $q$-Hermite polynomials. Now
let us change a bit this expansion by renaming its parameters. Let us denote
$a/b\allowbreak=\allowbreak\rho$. Then we can recognize that $\prod
_{k=0}^{\infty}(1-(1-q)axq^{k}+(1-q)a^{2}q^{2k})\allowbreak=\allowbreak
1/\varphi(x|a,q)$ and
\[
\left(  a^{2}/b^{2}\right)  _{\infty}\prod_{k=0}^{\infty}\frac{1}{w\left(
x,y,q^{k}a/b|q\right)  }=f_{CN}(x|y,a/b,q)/f_{N}(x|q).
\]
See also \cite{SzabP-M}.

\section{Complicated proofs and auxiliary facts from $q$-series
theory.\label{aux1}}

We start with the series of auxiliary, simplifying formulae.

\begin{lemma}
\label{Wm}The generating function of the sequence $\left\{  w_{n}%
(0,r_{1},r_{2},q)\right\}  $ is the following:%
\[
\sum_{n\geq0}\frac{t^{n}}{\left(  q\right)  _{n}}w_{n}(0,r_{1},r_{2}%
,q)=\frac{\left(  tr_{1}r_{2}^{2}\right)  _{\infty}\left(  tr_{2}r_{1}%
^{2}\right)  _{\infty}}{\left(  tr_{1}\right)  _{\infty}\left(  tr_{2}\right)
_{\infty}}.
\]
Consequently, we have, for all $n\geq0$, $r_{1},r_{2}\in(-1,1)$ and
$-1<q\leq1$:%
\begin{align}
w_{n}(m,r_{1},r_{2},q)  &  =q^{-nm/2}w_{n}(0,r_{1}q^{m/2},r_{2}q^{m/2}%
,q)\label{id1}\\
&  =\sum_{s=0}^{n}%
\genfrac{[}{]}{0pt}{}{n}{s}%
_{q}r_{1}^{n-s}r_{2}^{s}\left(  q^{m}r_{1}r_{2}\right)  _{s}\left(  q^{m}%
r_{1}r_{2}\right)  _{n-s}.\nonumber
\end{align}
and
\[
\sum_{n\geq0}\frac{t^{n}}{\left(  q\right)  _{n}}w_{n}(m,r_{1},r_{2}%
,q)=\frac{\left(  tq^{m}r_{1}r_{2}^{2}\right)  _{\infty}\left(  tq^{m}%
r_{2}r_{1}^{2}\right)  _{\infty}}{\left(  tr_{1}\right)  _{\infty}\left(
tr_{2}\right)  _{\infty}}%
\]

\end{lemma}

\begin{proof}
We have%
\[
\sum_{n\geq0}\frac{t^{n}}{\left(  q\right)  _{n}}w_{n}(0,r_{1},r_{2}%
,q)=\sum_{n\geq0}\sum_{s=0}^{n}\frac{\left(  tr_{1}\right)  ^{s}}{\left(
q\right)  _{s}}\left(  r_{2}^{2}\right)  _{s}\frac{\left(  tr_{2}\right)
^{n-s}}{\left(  q\right)  _{n-s}}\left(  r_{1}^{2}\right)  _{n-s}.
\]
Now, changing the order of summation we get%
\begin{align*}
\sum_{n\geq0}\frac{t^{n}}{\left(  q\right)  _{n}}w_{n}(0,r_{1},r_{2},q)  &
=\sum_{s\geq0}\frac{\left(  tr_{1}\right)  ^{s}}{\left(  q\right)  _{s}%
}\left(  r_{2}^{2}\right)  _{s}\sum_{n\geq s}\frac{\left(  tr_{2}\right)
^{n-s}}{\left(  q\right)  _{n-s}}\left(  r_{1}^{2}\right)  _{n-s}\\
&  =\frac{\left(  tr_{1}r_{2}^{2}\right)  _{\infty}\left(  tr_{2}r_{1}%
^{2}\right)  _{\infty}}{\left(  tr_{1}\right)  _{\infty}\left(  tr_{2}\right)
_{\infty}}.
\end{align*}
Further, we notice, that
\[
\frac{\left(  tr_{1}r_{2}^{2}\right)  _{\infty}\left(  tr_{2}r_{1}^{2}\right)
_{\infty}}{\left(  tr_{1}\right)  _{\infty}\left(  tr_{2}\right)  _{\infty}%
}=\frac{\left(  tr_{1}r_{2}^{2}\right)  _{\infty}\left(  tr_{2}r_{1}%
^{2}\right)  _{\infty}}{\left(  tr_{2}\right)  _{\infty}\left(  tr_{1}\right)
_{\infty}},
\]
from which follows directly (\ref{id1}).
\end{proof}

\begin{lemma}
\label{spec1}The following identity is true for all $n\geq0$, $r_{1},r_{2}%
\in(-1,1)$ and $-1<q\leq1$:%
\[
\sum_{s=0}^{n}%
\genfrac{[}{]}{0pt}{}{n}{s}%
_{q}r_{1}^{n-s}r_{2}^{s}\left(  r_{1}^{2}\right)  _{s}\left(  r_{1}%
r_{2}\right)  _{s}/\left(  r_{1}^{2}r_{2}^{2}\right)  _{s}=\frac{1}{\left(
r_{1}^{2}r_{2}^{2}\right)  _{n}}w_{n}(0,r_{1},r_{2},q)=\phi_{n}(r_{1}%
,r_{2},q).
\]

\end{lemma}

\begin{proof}
We will prove it by the generating function method. Firstly we notice that:
\[
\left(  r_{1}^{2}r_{2}^{2}\right)  _{n}\sum_{s=0}^{n}%
\genfrac{[}{]}{0pt}{}{n}{s}%
_{q}r_{1}^{n-s}r_{2}^{s}\left(  r_{1}^{2}\right)  _{s}\frac{\left(  r_{1}%
r_{2}\right)  _{s}}{\left(  r_{1}^{2}r_{2}^{2}\right)  _{s}}=\sum_{s=0}^{n}%
\genfrac{[}{]}{0pt}{}{n}{s}%
_{q}r_{1}^{n-s}r_{2}^{s}\left(  r_{1}^{2}\right)  _{s}\left(  r_{1}%
r_{2}\right)  _{s}\left(  r_{1}^{2}r_{2}^{2}q^{s}\right)  _{n-s}.
\]
Secondly, we calculate the generating function of the sequence \newline%
$\left\{  \sum_{s=0}^{n}%
\genfrac{[}{]}{0pt}{}{n}{s}%
_{q}r_{1}^{n-s}r_{2}^{s}\left(  r_{1}^{2}\right)  _{s}\left(  r_{1}%
r_{2}\right)  _{s}\left(  r_{1}^{2}r_{2}^{2}q^{s}\right)  _{n-s}\right\}
_{n\geq0}$. We have
\begin{align*}
&  \sum_{n\geq0}\frac{t^{n}}{\left(  q\right)  _{n}}\sum_{s=0}^{n}%
\genfrac{[}{]}{0pt}{}{n}{s}%
_{q}r_{1}^{n-s}r_{2}^{s}\left(  r_{1}^{2}\right)  _{s}\left(  r_{1}%
r_{2}\right)  _{s}\left(  r_{1}^{2}r_{2}^{2}q^{s}\right)  _{n-s}\\
&  =\sum_{s\geq0}\frac{\left(  tr_{2}\right)  ^{s}}{\left(  q\right)  _{s}%
}\left(  r_{1}^{2}\right)  _{s}\left(  r_{1}r_{2}\right)  _{s}\sum_{n\geq
s}\frac{\left(  tr_{1}\right)  ^{n-s}}{\left(  q\right)  _{n-s}}\left(
r_{1}^{2}r_{2}^{2}q^{s}\right)  _{n-s}\\
&  =\sum_{s\geq0}\frac{\left(  tr_{2}\right)  ^{s}}{\left(  q\right)  _{s}%
}\left(  r_{1}^{2}\right)  _{s}\left(  r_{1}r_{2}\right)  _{s}\frac{\left(
tr_{1}^{3}r_{2}^{2}q^{s}\right)  _{\infty}}{\left(  tr_{1}\right)  _{\infty}%
}\\
&  =\frac{\left(  tr_{1}^{3}r_{2}^{2}\right)  _{\infty}}{\left(
tr_{1}\right)  _{\infty}}\sum_{s\geq0}\frac{\left(  tr_{2}\right)  ^{s}%
}{\left(  q\right)  _{s}}\frac{\left(  r_{1}^{2}\right)  _{s}\left(
r_{1}r_{2}\right)  _{s}}{\left(  tr_{1}^{3}r_{2}^{2}\right)  _{s}}%
=\frac{\left(  tr_{1}^{3}r_{2}^{2}\right)  _{\infty}}{\left(  tr_{1}\right)
_{\infty}}~_{2}\phi_{1}\left(
\begin{array}
[c]{cc}%
r_{1}^{2} & r_{1}r_{2}\\
tr_{1}^{3}r_{2}^{2} &
\end{array}
;q,tr_{2}\right)  ,
\end{align*}
where $_{2}\phi_{1}$ denotes the so-called basic hypergeometric function (see,
e.g., \cite{KLS}(1.10.1)) (different from the function defined by (\ref{fi})).
Reading its properties, in particular, the so-called reduction formulae, we
observe that:
\begin{equation}
tr_{2}=\frac{tr_{1}^{3}r_{2}^{2}}{r_{1}^{2}r_{1}r_{2}}. \label{upr}%
\end{equation}
So now we use the so-called $q$-Gauss sum. It is one of the reduction formulae
for $_{2}\phi_{1}$ presented in \cite{KLS}. Namely the one given by (1.13.2)
with (\ref{upr}) being equivalent to $ab/c\allowbreak=\allowbreak z^{-1}$. We
get then
\[
_{2}\phi_{1}\left(
\begin{array}
[c]{cc}%
r_{1}^{2} & r_{1}r_{2}\\
tr_{1}^{3}r_{2}^{2} &
\end{array}
;q,tr_{2}\right)  =\frac{\left(  tr_{1}^{3}r_{2}^{2}/r_{1}^{2}\right)
_{\infty}\left(  tr_{1}^{3}r_{2}^{2}/r_{1}r_{2}\right)  _{\infty}}{\left(
tr_{1}^{3}r_{2}^{2}\right)  _{\infty}\left(  tr_{2}\right)  _{\infty}}.
\]
Hence, indeed we have
\[
\frac{\left(  tr_{1}^{3}r_{2}^{2}\right)  _{\infty}}{\left(  tr_{1}\right)
_{\infty}}~_{2}\phi_{1}\left(
\begin{array}
[c]{cc}%
r_{1}^{2} & r_{1}r_{2}\\
tr_{1}^{3}r_{2}^{2} &
\end{array}
;q,tr_{2}\right)  =\frac{\left(  tr_{1}r_{2}^{2}\right)  _{\infty}\left(
tr_{2}r_{1}^{2}\right)  _{\infty}}{\left(  tr_{1}\right)  _{\infty}\left(
tr_{2}\right)  _{\infty}}.
\]
Now we recall the assertion on the Lemma \ref{Wm}.
\end{proof}

\begin{lemma}
\label{pom1}The following identity is true for all $n\geq0$, $r_{1},r_{2}%
\in(-1,1)$ and $-1<q\leq1$:%
\[
\sum_{j=0}^{m}%
\genfrac{[}{]}{0pt}{}{m}{j}%
_{q}(-1)^{m-j}q^{\binom{m-j}{2}}\frac{\left(  aq^{j}\right)  _{m}}{1-aq^{j}}=%
\begin{cases}
1/(1-a)\text{,} & \text{if }m=0\text{;}\\
0, & \text{if }m>0\text{.}%
\end{cases}
\]

\end{lemma}

\begin{proof}
It is obvious that when $m\allowbreak=\allowbreak0$, the identity is true.
Hence, let us assume that $m\geq1$. Now we change the index of summation from
$j$ to $t\allowbreak=\allowbreak m-j$ and apply (\ref{naw}). We get then
\begin{gather*}
\sum_{j=0}^{m}%
\genfrac{[}{]}{0pt}{}{m}{j}%
_{q}(-1)^{j}q^{\binom{j}{2}}\frac{\left(  aq^{m-j}\right)  _{m}}{1-aq^{m-j}%
}=\sum_{j=0}^{m}%
\genfrac{[}{]}{0pt}{}{m}{j}%
_{q}(-1)^{j}q^{\binom{j}{2}}(aq^{m-j})_{m-1}\\
=\sum_{j=0}^{m}%
\genfrac{[}{]}{0pt}{}{m}{j}%
_{q}(-1)^{j}q^{\binom{j}{2}}\sum_{k=0}^{m-1}%
\genfrac{[}{]}{0pt}{}{m-1}{k}%
_{q}(-aq^{m-j})^{k}q^{\binom{k}{2}}\\
=\sum_{k=0}^{m-1}%
\genfrac{[}{]}{0pt}{}{m-1}{k}%
_{q}(-a)^{k}q^{m-j})^{k}q^{\binom{k}{2}}\sum_{j=0}^{m}%
\genfrac{[}{]}{0pt}{}{m}{j}%
_{q}(-1)^{j}q^{\binom{j}{2}}q^{k(m-j)}\\
=\sum_{k=0}^{m-1}%
\genfrac{[}{]}{0pt}{}{m-1}{k}%
_{q}(-a)^{k}q^{m-j})^{k}q^{\binom{k}{2}}q^{km}\sum_{j=0}^{m}%
\genfrac{[}{]}{0pt}{}{m}{j}%
_{q}(-1)^{j}q^{\binom{j}{2}}q^{-kj}\\
=\sum_{k=0}^{m-1}%
\genfrac{[}{]}{0pt}{}{m-1}{k}%
_{q}(-a)^{k}q^{m-j})^{k}q^{\binom{k}{2}}q^{km}(q^{-k})_{m}=0.
\end{gather*}
This is so since $\left(  q^{-k}\right)  _{m}\allowbreak=\allowbreak0$ for
every $k\allowbreak=\allowbreak0,\ldots,m-1$.
\end{proof}

\begin{lemma}
\label{skrt}The following identity is true for all $n,t\geq0$, $r_{1},r_{2}%
\in(-1,1)$ and $-1<q\leq1$:%
\begin{equation}
\sum_{k=0}^{m}%
\genfrac{[}{]}{0pt}{}{m}{k}%
_{q}(-r_{2}^{2})^{k}q^{\binom{k}{2}}\frac{\left(  r_{1}^{2}\right)  _{t+m+k}%
}{\left(  r_{1}^{2}r_{2}^{2}\right)  _{t+m+k}}=\frac{\left(  r_{2}^{2}\right)
_{m}\left(  r_{1}^{2}\right)  _{t+m}}{\left(  r_{1}^{2}r_{2}^{2}\right)
_{t+2m}}. \label{id2}%
\end{equation}

\end{lemma}

\begin{proof}
Recall, that in \cite{Szab6} (Lemma 1ii) ) the following identity has been
proved for all $-1<q\leq1$, $a,b\in\mathbb{R}$ and $n\geq0$:%
\begin{equation}
\sum_{j=0}^{n}%
\genfrac{[}{]}{0pt}{}{n}{j}%
_{q}(-b)^{j}q^{\binom{j}{2}}\left(  a\right)  _{j}\left(  abq^{j}\right)
_{n-j}\allowbreak=\allowbreak\left(  b\right)  _{n}. \label{aux2}%
\end{equation}

So now let us transform, a bit, the identity that we must prove. Namely, after
multiplying both sides by $\left(  r_{1}^{2}r_{2}^{2}\right)  _{t+m}$ and
dividing again both sides by $\left(  r_{1}^{2}\right)  _{t+m}$ we get%
\[
\sum_{k=0}^{m}%
\genfrac{[}{]}{0pt}{}{m}{k}%
_{q}(-r_{2}^{2})^{k}q^{\binom{k}{2}}\left(  r_{1}^{2}q^{t+m}\right)
_{k}\left(  r_{1}^{2}r_{2}^{2}q^{t+m}\right)  _{m-k}.
\]
We apply (\ref{aux2}) with $a\allowbreak=\allowbreak r_{1}^{2}q^{t+m}$ and
$b\allowbreak=\allowbreak r_{2}^{2}$ and get (\ref{id2}).
\end{proof}

\begin{proof}
[Proof of Proposition \ref{rozw}]We start with inserting (\ref{HRnaR}) into
(\ref{Dn}) and getting:%
\begin{gather*}
D_{n}(y|r_{1},r_{2},q)=\sum_{s=0}^{n}%
\genfrac{[}{]}{0pt}{}{n}{s}%
_{q}r_{1}^{n-s}r_{2}^{s}\frac{\left(  r_{1}^{2}\right)  _{s}}{\left(
r_{1}^{2}r_{2}^{2}\right)  _{s}}\times\\
\sum_{u=0}^{\left\lfloor n/2\right\rfloor }\frac{\left[  s\right]
_{q}!\left[  n-s\right]  _{q}!(1-r_{1}r_{2}q^{n-2u})}{\left[  u\right]
_{q}!\left[  n-2u\right]  _{q}!}R_{n-2u}(y|r_{1}r_{2},q)\times\\
\sum_{m=0}^{u}%
\genfrac{[}{]}{0pt}{}{u}{m}%
_{q}\frac{(r_{1}r_{2})^{u-m}}{(1-r_{1}r_{2})\left(  r_{1}r_{2}q\right)
_{n-m-u}}\sum_{k=0}^{m}%
\genfrac{[}{]}{0pt}{}{m}{k}%
_{q}%
\genfrac{[}{]}{0pt}{}{n-2m}{n+k-s-m}%
_{q}(-r_{1}r_{2})^{k}q^{\binom{k}{2}}\left(  r_{1}r_{2}\right)  _{s-k}%
\end{gather*}

First, we change the order of summation and get
\begin{gather*}
D_{n}(y|r_{1},r_{2},q)=\sum_{u=0}^{\left\lfloor n/2\right\rfloor }%
\frac{\left[  n\right]  _{q}!(1-r_{1}r_{2}q^{n-2u})}{\left[  u\right]
_{q}!\left[  n-2u\right]  _{q}!\left(  r_{1}r_{2}\right)  _{n-u+1}}%
R_{n-2u}(y|r_{1}r_{2},q)\sum_{s=0}^{n}r_{1}^{n-s}r_{2}^{s}\frac{\left(
r_{1}^{2}\right)  _{s}}{\left(  r_{1}^{2}r_{2}^{2}\right)  _{s}}\\
\times\sum_{m=0}^{u}%
\genfrac{[}{]}{0pt}{}{u}{m}%
_{q}\left(  r_{1}r_{2}\right)  ^{u-m}\left(  r_{1}r_{2}q^{n+1-u-m}\right)
_{m}\times\\
\sum_{k=0}^{m}%
\genfrac{[}{]}{0pt}{}{m}{k}%
_{q}%
\genfrac{[}{]}{0pt}{}{n-2m}{n+k-s-m}%
_{q}(-r_{1}r_{2})^{k}q^{\binom{k}{2}}\left(  r_{1}r_{2}\right)  _{s-k}.
\end{gather*}
On the way, we used the well-known property of the $q$-Pochhammer symbol (see,
e.g., \cite{KLS}):%
\[
\left(  a\right)  _{n+m}=\left(  a\right)  _{n}\left(  aq^{n}\right)  _{m}.
\]

Now notice, that for $%
\genfrac{[}{]}{0pt}{}{n-2m}{n-s-(m-k)}%
_{q}$ to be nonzero we have to have $n-s\geq m-k$ and $n-2m-n+s+m-k\geq
0,\ $which leads to $s\geq m+k$. Hence, we have further:%
\begin{gather*}
D_{n}(y|r_{1},r_{2},q)=\sum_{u=0}^{\left\lfloor n/2\right\rfloor }%
\frac{\left[  n\right]  _{q}!(1-r_{1}r_{2}q^{n-2u})}{\left[  u\right]
_{q}!\left[  n-2u\right]  _{q}!\left(  r_{1}r_{2}\right)  _{n-u+1}}%
R_{n-2u}(y|r_{1}r_{2},q)\times\\
\sum_{m=0}^{u}%
\genfrac{[}{]}{0pt}{}{u}{m}%
_{q}(r_{1}r_{2})^{u-m}\left(  r_{1}r_{2}q^{n-u-m+1}\right)  _{m}\times\\
\sum_{k=0}^{m}%
\genfrac{[}{]}{0pt}{}{m}{k}%
_{q}(-r_{1}r_{2})^{k}q^{\binom{k}{2}}\sum_{s=m+k}^{n-(m-k)}%
\genfrac{[}{]}{0pt}{}{n-2m}{n-s-(m-k)}%
_{q}r_{1}^{n-s}r_{2}^{s}\frac{\left(  r_{1}^{2}\right)  _{s}}{\left(
r_{1}^{2}r_{2}^{2}\right)  _{s}}\left(  r_{1}r_{2}\right)  _{s-k}\\
=\sum_{u=0}^{\left\lfloor n/2\right\rfloor }\frac{\left[  n\right]
_{q}!(1-r_{1}r_{2}q^{n-2u})}{\left[  u\right]  _{q}!\left[  n-2u\right]
_{q}!\left(  r_{1}r_{2}\right)  _{n-u+1}}R_{n-2u}(y|r_{1}r_{2},q)\times\\
\sum_{m=0}^{u}%
\genfrac{[}{]}{0pt}{}{u}{m}%
_{q}(r_{1}r_{2})^{u-m}\left(  r_{1}r_{2}q^{n-u-m+1}\right)  _{m}\sum_{k=0}^{m}%
\genfrac{[}{]}{0pt}{}{m}{k}%
_{q}(-r_{1}r_{2})^{k}q^{\binom{k}{2}}\\
\times\sum_{t=0}^{n-2m}%
\genfrac{[}{]}{0pt}{}{n-2m}{t}%
_{q}r_{1}^{n-t-m-k}r_{2}^{t+m+k}\frac{\left(  r_{1}^{2}\right)  _{t+m+k}%
}{\left(  r_{1}^{2}r_{2}^{2}\right)  _{t+m+k}}\left(  r_{1}r_{2}\right)
_{t+m}.
\end{gather*}
Further, we get%
\begin{gather*}
D_{n}(y|r_{1},r_{2},q)=\sum_{u=0}^{\left\lfloor n/2\right\rfloor }%
\frac{\left[  n\right]  _{q}!(1-r_{1}r_{2}q^{n-2u})}{\left[  u\right]
_{q}!\left[  n-2u\right]  _{q}!\left(  r_{1}r_{2}\right)  _{n-u+1}}%
R_{n-2u}(y|r_{1}r_{2},q)\times\\
\sum_{m=0}^{u}%
\genfrac{[}{]}{0pt}{}{u}{m}%
_{q}(r_{1}r_{2})^{u-m}\left(  r_{1}r_{2}q^{n-u-m+1}\right)  _{m}\times\\
\sum_{k=0}^{m}%
\genfrac{[}{]}{0pt}{}{m}{k}%
_{q}(-r_{1}^{2})^{k}q^{\binom{k}{2}}\sum_{t=0}^{n-2m}%
\genfrac{[}{]}{0pt}{}{n-2m}{t}%
_{q}r_{1}^{n-t-2m}r_{2}^{t}\frac{\left(  r_{1}^{2}\right)  _{t+m+k}}{\left(
r_{1}^{2}r_{2}^{2}\right)  _{t+m+k}}\left(  r_{1}r_{2}\right)  _{t+m}.
\end{gather*}

Now, we change the order of summation in the last two sums, and we get%

\begin{gather*}
D_{n}(y|r_{1},r_{2},q)=\sum_{u=0}^{\left\lfloor n/2\right\rfloor }%
\frac{\left[  n\right]  _{q}!(1-r_{1}r_{2}q^{n-2u})}{\left[  u\right]
_{q}!\left[  n-2u\right]  _{q}!\left(  r_{1}r_{2}\right)  _{n-u+1}}%
R_{n-2u}(y|r_{1}r_{2},q)\times\\
\sum_{m=0}^{u}%
\genfrac{[}{]}{0pt}{}{u}{m}%
_{q}(r_{1}r_{2})^{u-m}\left(  r_{1}r_{2}q^{n-u-m+1}\right)  _{m}\times\\
\sum_{t=0}^{n-2m}%
\genfrac{[}{]}{0pt}{}{n-2m}{n-2m-t}%
_{q}r_{1}^{n-t-2m}r_{2}^{t}\left(  r_{1}r_{2}\right)  _{t+m}\sum_{k=0}^{m}%
\genfrac{[}{]}{0pt}{}{m}{k}%
_{q}(-r_{1}^{2})^{k}q^{\binom{k}{2}}\frac{\left(  r_{1}^{2}\right)  _{t+m+k}%
}{\left(  r_{1}^{2}r_{2}^{2}\right)  _{t+m+k}}.
\end{gather*}
Now, notice, that after applying Lemma \ref{skrt} we get%
\begin{gather*}
D_{n}(y|r_{1},r_{2},q)=\sum_{u=0}^{\left\lfloor n/2\right\rfloor }%
\frac{\left[  n\right]  _{q}!(1-r_{1}r_{2}q^{n-2u})(r_{1}r_{2})^{u}}{\left[
u\right]  _{q}!\left[  n-2u\right]  _{q}!\left(  r_{1}r_{2}\right)  _{n-u+1}%
}R_{n-2u}(y|r_{1}r_{2},q)\times\\
\sum_{m=0}^{u}%
\genfrac{[}{]}{0pt}{}{u}{m}%
_{q}\left(  r_{2}^{2}\right)  _{m}\left(  r_{1}r_{2}q^{n-u-m+1}\right)  _{m}\\
\times\sum_{t=0}^{n-2m}%
\genfrac{[}{]}{0pt}{}{n-2m}{n-2m-t}%
_{q}r_{1}^{n-t-2m}r_{2}^{t}\left(  r_{1}r_{2}\right)  _{t+m}\frac{\left(
r_{1}^{2}\right)  _{t+m}}{\left(  r_{1}^{2}r_{2}^{2}\right)  _{t+2m}}.
\end{gather*}
Notice that
\begin{gather*}
\sum_{t=0}^{n-2m}%
\genfrac{[}{]}{0pt}{}{n-2m}{n-2m-t}%
_{q}r_{1}^{n-t-2m}r_{2}^{t}\left(  r_{1}r_{2}\right)  _{t+m}\frac{\left(
r_{1}^{2}\right)  _{t+m}}{\left(  r_{1}^{2}r_{2}^{2}\right)  _{t+2m}%
}\allowbreak=\\
\allowbreak\sum_{s=0}^{n-2m}%
\genfrac{[}{]}{0pt}{}{n-2m}{s}%
_{q}r_{1}^{s}r_{2}^{n-2m-s}\left(  r_{1}r_{2}\right)  _{n-m-s}\frac{\left(
r_{1}^{2}\right)  _{n-m-s}}{\left(  r_{1}^{2}r_{2}^{2}\right)  _{n-s}}=\\
\frac{\left(  r_{1}^{2}\right)  _{m}\left(  r_{1}r_{2}\right)  _{m}}{\left(
r_{1}^{2}r_{2}^{2}\right)  _{2m}}\sum_{s=0}^{n-2m}%
\genfrac{[}{]}{0pt}{}{n-2m}{s}%
_{q}r_{1}^{s}r_{2}^{n-2m-s}\frac{\left(  r_{1}r_{2}q^{m}\right)
_{n-2m-s}\left(  r_{1}^{2}q^{m}\right)  _{n-m-s}}{\left(  r_{1}^{2}r_{2}%
^{2}q^{2m}\right)  _{n-s}}%
\end{gather*}
and further that%
\begin{gather*}
=\frac{\left(  r_{1}^{2}\right)  _{m}\left(  r_{1}r_{2}\right)  _{m}}{\left(
r_{1}^{2}r_{2}^{2}\right)  _{2m}}q^{-m(n-2m)/2}\times\\
\sum_{s=0}^{n-2m}%
\genfrac{[}{]}{0pt}{}{n-2m}{s}%
_{q}(q^{m/2}r_{1})^{s}(q^{m/2}r_{2})^{n-2m-s}\frac{\left(  r_{1}r_{2}%
q^{m}\right)  _{n-2m-s}\left(  r_{1}^{2}q^{m}\right)  _{n-m-s}}{\left(
r_{1}^{2}r_{2}^{2}q^{2m}\right)  _{n-s}}\\
=\frac{\left(  r_{1}^{2}\right)  _{m}\left(  r_{1}r_{2}\right)  _{m}}{\left(
r_{1}^{2}r_{2}^{2}\right)  _{2m}}\frac{1}{\left(  r_{1}^{2}r_{2}^{2}%
q^{2m}\right)  _{n-2n}}\sum_{s=0}^{n-2m}%
\genfrac{[}{]}{0pt}{}{n-2m}{s}%
_{q}r_{1}^{s}r_{2}^{n-s}\left(  r_{2}^{2}q^{m}\right)  _{s}\left(  r_{1}%
^{2}q^{m}\right)  _{n-s}\\
=\frac{\left(  r_{1}^{2}\right)  _{m}\left(  r_{1}r_{2}\right)  _{m}}{\left(
r_{1}^{2}r_{2}^{2}\right)  _{n}}w_{n-2m}(m,r_{1},r_{2},q).
\end{gather*}

At the final stage, we used Lemma \ref{spec1} and the identity $\left(
a\right)  _{n+m}\allowbreak=\allowbreak\left(  a\right)  _{n}\left(
aq^{n}\right)  _{m}$ multiple times.

Concluding, we have%
\begin{gather*}
D_{n}(y|r_{1},r_{2},q)=\frac{1}{\left(  r_{1}^{2}r_{2}^{2}\right)  _{n}}%
\sum_{u=0}^{\left\lfloor n/2\right\rfloor }\frac{\left[  n\right]
_{q}!(1-r_{1}r_{2}q^{n-2u})(r_{1}r_{2})^{u}}{\left[  u\right]  _{q}!\left[
n-2u\right]  _{q}!\left(  r_{1}r_{2}\right)  _{n-u+1}}R_{n-2u}(y|r_{1}%
r_{2},q)\times\\
\sum_{m=0}^{u}%
\genfrac{[}{]}{0pt}{}{u}{m}%
_{q}\left(  r_{2}^{2}\right)  _{m}\left(  r_{1}r_{2}q^{n-u-m+1}\right)
_{m}\left(  r_{1}^{2}\right)  _{m}\left(  r_{1}r_{2}\right)  _{m}%
w_{n-2m}(m,r_{1}r_{2},q).
\end{gather*}
Or after applying the definition of $\gamma_{n,k}(r_{1},r_{2},q)$ given by
(\ref{fin})%
\[
D_{n}(y|r_{1},r_{2},q)=\sum_{u=0}^{\left\lfloor n/2\right\rfloor }%
\frac{\left[  n\right]  _{q}!(1-r_{1}r_{2}q^{n-2u})(r_{1}r_{2})^{u}}{\left[
u\right]  _{q}!\left[  n-2u\right]  _{q}!\left(  r_{1}r_{2}\right)  _{n-u+1}%
}R_{n-2u}(y|r_{1}r_{2},q)\gamma_{n,u}(r_{1},r_{2},q).
\]

\end{proof}

\begin{proof}
[Proof of Lemma \ref{final}.]Let us denote for brevity $K(x,y)\allowbreak
=\allowbreak(1-r_{1}r_{2})f_{CN}\left(  y|x,r_{1},q\right)  \allowbreak
\times\allowbreak f_{CN}\left(  x|y,r_{2},q\right)  $. As shown in the Remark
\ref{sym}, function $K$ is a symmetric function of $x$ and $y$. Now imagine
that we multiply function $K$ by any function $g(x)$, and integrate the
product over $S(q)$ with respect to $x$. Let us call the result $h(y)$. Now
imagine that we multiply $K(x,y)$ by $g(y)$ and integrate with respect to $y$
over $S(q)$. We should get $h(x)$.

The proof will be by induction with respect to $s\allowbreak=$\allowbreak
$n-2u$. The induction assumption is the following:
\begin{align*}
\gamma_{n,k}(r_{1},r_{2},q)\allowbreak &  =\allowbreak\phi_{n-2k}(r_{1}%
,r_{2},q)\text{ and }\\
\int_{S(q)}K(x,y)R_{n}(x|r_{1}r_{2},q)dx  &  =\phi_{n}(r_{1},r_{2}%
,q)R_{n}(y|r_{1}r_{2},q).
\end{align*}

So, now to start induction, let us set $s\allowbreak=\allowbreak0$.
Integrating the right-hand side (\ref{exp1}) with respect to $x$ yields
$f_{R}(y|r_{1}r_{2},q)$. Now, integration of the right-hand side of
(\ref{exp1}) with respect to $y$ results in :%
\begin{gather*}
f_{N}(x|q)\sum_{n\geq0}\frac{H_{2n}(x|q)}{\left[  2n\right]  _{q}!}\int%
_{S(q)}D_{2n}(y|r_{1},r_{2},q)f_{R}(y|r_{1}r_{2},q)dy=\\
f_{N}(x|q)\sum_{n\geq0}\frac{H_{2n}(x|q)}{\left[  2n\right]  _{q}!}%
\frac{(1-r_{1}r_{2})(r_{1}r_{2})^{n}\left[  2n\right]  _{q}!}{\left[
n\right]  _{q}!(r_{1}r_{2})_{n+1}}\gamma_{2n,n}(r_{1},r_{2},q)\\
=(1-r_{1}r_{2})f_{N}(x|q)\sum_{n\geq0}\frac{H_{2n}(x|q)}{\left[  n\right]
_{q}!}\frac{(r_{1}r_{2})^{n}}{(r_{1}r_{2})_{n+1}}\gamma_{2n,n}(r_{1},r_{2},q).
\end{gather*}
Hence, using Proposition \ref{aux1} and the uniqueness of the expansion in
orthogonal polynomials, we must have
\[
\gamma_{2n,n}(r_{1},r_{2},q)=\phi_{0}(r_{1},r_{2},1)=1,
\]
for all $n\geq0$.

So now let us take $s\allowbreak=\allowbreak m$ and make the induction
assumption that $\gamma_{n,k}\left(  r_{1},r_{2},q\right)  \allowbreak
=\allowbreak$ $\phi_{n-2k}(r_{1},r_{2},q)$ whenever $n-2k<m$.

Now let us multiply the left-hand side of (\ref{exp1}) by $R_{m}(x|r_{1}%
r_{2},q)$ and integrate over $S\left(  q\right)  $ with respect to $x$. Since
(\ref{HRfN}) is zero for $n>m$, we have%
\begin{gather*}
f_{R}(y|r_{1}r_{2},q)\sum_{s=0}^{\left\lfloor m/2\right\rfloor }\frac
{D_{m-2s}(y|r_{1},r_{2},q)}{\left[  m-2s\right]  _{q}!}\int_{S(q)}%
H_{m-2s}(x|q)R_{m}(x|r_{1}r_{2},q)f_{N}(x|q)dx=\\
f_{R}(y|r_{1}r_{2},q)\sum_{s=0}^{\left\lfloor m/2\right\rfloor }(-r_{1}%
r_{2})^{s}q^{\binom{s}{2}}\frac{\left[  m\right]  _{q}!\left(  r_{1}%
r_{2}\right)  _{m-s}}{\left[  s\right]  _{q}!}\frac{1}{\left[  m-2s\right]
_{q}!}\times\\
\sum_{u=0}^{\left\lfloor m/2\right\rfloor -s}\frac{\left[  m-2s\right]
_{q}!(1-r_{1}r_{2}q^{m-2s-2u})(r_{1}r_{2})^{u}}{\left[  u\right]  _{q}!\left[
m-2s-2u\right]  _{q}!(r_{1}r_{2})_{m-2s-u+1}}R_{m-2s-2u}(r_{1},r_{2}%
,q)\gamma_{m-2s,u}(r_{1},r_{2},q)=\\
f_{R}(y|r_{1}r_{2},q)\sum_{k=0}^{\left\lfloor m/2\right\rfloor }\frac{\left[
m\right]  _{q}!(1-r_{1}r_{2}q^{m-2k})(r_{1}r_{2})^{k}}{\left[  k\right]
_{q}!\left[  m-2k\right]  _{q}!}R_{m-2k}(r_{1},r_{2},q)\gamma_{m-2k,0}%
(r_{1},r_{2},q)\\
\times\sum_{s=0}^{k}%
\genfrac{[}{]}{0pt}{}{k}{s}%
_{q}(-1)^{s}q^{\binom{s}{2}}\frac{\left(  r_{1}r_{2}\right)  _{m-s}}%
{(r_{1}r_{2})_{m-k-s+1}}=
\end{gather*}%
\begin{gather*}
f_{R}(y|r_{1}r_{2},q)\sum_{k=0}^{\left\lfloor m/2\right\rfloor }\frac{\left[
m\right]  _{q}!(1-r_{1}r_{2}q^{m-2k})(r_{1}r_{2})^{k}}{\left[  k\right]
_{q}!\left[  m-2k\right]  _{q}!}R_{m-2k}(r_{1},r_{2},q)\gamma_{m-2k,0}%
(r_{1},r_{2},q)\\
\times\sum_{u=0}^{k}%
\genfrac{[}{]}{0pt}{}{k}{u}%
_{q}(-1)^{k-u}q^{\binom{k-u}{2}}\frac{\left(  r_{1}r_{2}\right)  _{m-k+u}%
}{(r_{1}r_{2})_{m-2k+u+1}}=\\
f_{R}(y|r_{1}r_{2},q)\sum_{k=0}^{\left\lfloor m/2\right\rfloor }\frac{\left[
m\right]  _{q}!(1-r_{1}r_{2}q^{m-2k})(r_{1}r_{2})^{k}}{\left[  k\right]
_{q}!\left[  m-2k\right]  _{q}!}R_{m-2k}(r_{1},r_{2},q)\gamma_{m-2k,0}%
(r_{1},r_{2},q)\\
\times\sum_{u=0}^{k}%
\genfrac{[}{]}{0pt}{}{k}{u}%
_{q}(-1)^{k-u}q^{\binom{k-u}{2}}\frac{\left(  r_{1}r_{2}q^{m-2k+u}\right)
_{k}}{(1-r_{1}r_{2}q^{m-2k+u+1})}=\\
f_{R}(y|r_{1}r_{2},q)R_{m}(r_{1},r_{2},q)\phi_{m}(r_{1},r_{2},q).
\end{gather*}

We applied here induction assumption as well as Lemma \ref{pom1} with
$a\allowbreak=\allowbreak r_{1}r_{2}q^{m-2k}$.

Now, let us multiply the (\ref{exp1}) by $R_{m}(y|r_{1}r_{2},q)$ and integrate
with respect to $y$ over $S(q)$, We get then
\begin{gather*}
f_{N}(x|q)\sum_{n\geq0}\frac{1}{\left[  n\right]  _{q}!}H_{n}(x|q)\int%
_{S(q)}D_{n}(y|r_{1},r_{2},q)R_{m}f_{R}(y|r_{1}r_{2},q)dy=\\
f_{N}(x|q)\sum_{u\geq0}\frac{H_{2u+m}(x|q)\left[  2u+m\right]  _{q}%
!(1-r_{1}r_{2}q^{m})\left[  m\right]  _{q}!(r_{1}^{2}r_{2}^{2})_{m}%
(1-r_{1}r_{2})}{\left[  2u+m\right]  _{q}!\left[  u\right]  _{q}!\left[
m\right]  _{q}!\left(  r_{1}r_{2}\right)  _{m+u+1}(1-r_{1}r_{2}q^{m})}=\\
(1-r_{1}r_{2})f_{N}(x|q)(r_{1}^{2}r_{2}^{2})_{m}\sum_{u\geq0}\frac
{H_{2u+m}(x|q)}{\left[  u\right]  _{q}!\left(  r_{1}r_{2}\right)  _{m+u+1}%
}\gamma_{2u+m,u}(r_{1},r_{2},q).
\end{gather*}
Now, in order to have this expression to be equal to $R_{m}(x|r_{1}%
r_{2},q)f_{R}(x|r_{1}r_{2},q)$, in the face of Proposition1(3) of
\cite{Szab19}, we must have
\[
\gamma_{2u+m,u}(r_{1},r_{2},q)=\phi_{m}(r_{1},r_{2},q).
\]

\end{proof}

\end{document}